\documentclass[11pt]{article}

\usepackage{amsfonts}
\usepackage{amsmath}
\usepackage{amssymb}
\usepackage{amsfonts}
\usepackage{amsmath}

\usepackage{graphicx}
\usepackage{ifthen}
\usepackage{comment}
\usepackage{xcolor}
\usepackage{braket}
\usepackage{enumitem}
\usepackage{lipsum,multicol}

\usepackage[T1]{fontenc}
\usepackage[english]{babel}
\usepackage{authblk}
\setlist[itemize]{noitemsep, topsep=0pt}
\setlist[enumerate]{noitemsep, topsep=0pt}
\setlist[itemize]{leftmargin=*}
\setlist[enumerate]{leftmargin=*}
\providecommand{\U}[1]{\protect\rule{.1in}{.1in}}

\providecommand{\norm}[1]{\left\lVert#1\right\rVert}

\providecommand{\pr}[1]{\left(#1\right)}

\providecommand{\set}[1]{\left\lbrace#1\right\rbrace}

\newcommand{\normi}[3]{\norm{#1}_{
    \ifthenelse{\equal{#2}{1}}{H_0^1\pr{\mathcal{O}_{#3}}}{%
    \ifthenelse{\equal{#2}{-1}}{H^{-1}\pr{\mathcal{O}_{#3}}}{}}}}

\makeatletter
\newcommand{\subjclass}[2][2020]{%
  \let\@oldtitle\@title%
  \gdef\@title{\@oldtitle\footnotetext{\textbf{#1 \emph{Mathematics subject classification.}} #2}}%
}
\newcommand{\keywords}[1]{%
  \let\@@oldtitle\@title%
  \gdef\@title{\@@oldtitle\footnotetext{\textbf{\emph{Key words.}} #1.}}%
}
\makeatother

\newtheorem{theorem}{Theorem}

\newtheorem{definition}[theorem]{Definition}

\newtheorem{lemma}[theorem]{Lemma}

\newtheorem{remark}[theorem]{Remark}

\newenvironment{proof}[1][Proof]{\noindent\textbf{#1.} }{\ \rule{0.5em}{0.5em}}

\begin{document}

\title{Backward stochastic porous media equation with polynomial growth}

\author{{
Ioana Ciotir \textsuperscript{a}\thanks{I. C. email: ioana.ciotir@insa-rouen.fr}, Dan Goreac\textsuperscript{b,c,d}\thanks{D. Goreac email: dan.goreac@act.ulaval.ca}, Juan Li \textsuperscript{c,e}\thanks{J. L. email: juanli@sdu.edu.cn}, Yufan Peng \textsuperscript{c}\thanks{Y. P. email: yufanpeng@mail.sdu.edu.cn} ,
Eduard Rotenstein\textsuperscript{f}\thanks{E. R. email: eduard.rotenstein@uaic.ro}}}
\affil{\footnotesize \textsuperscript{a} {Normandie University, INSA de Rouen Normandie, LMI (EA 3226 – FR CNRS 3335), 76000 Rouen, France;}\\
\textsuperscript{b} {\'{E}cole d'Actuariat, Université Laval, Québec, G1V 0A6, Canada;} \\
\textsuperscript{c} {School of Mathematics and Statistics, Shandong University, Weihai, Weihai 264209, PR China;}\\
\textsuperscript{d} {LAMA, Univ Gustave Eiffel, UPEM, Univ Paris Est Creteil, CNRS, F-77447 Marne-la-vall\'{e}e, France;}\\
\textsuperscript{e}{Research Center for Mathematics and Interdisciplinary Sciences, Shandong University, Qingdao 266237, PR China};\\
\textsuperscript{f} {Faculty of Mathematics, "Alexandru Ioan Cuza" University of Iasi, 700506, Romania.}
}
\date{October 15, 2026}
\maketitle
\begin{abstract} We investigate the well-posedness of a backward stochastic porous media equation whose nonlinear term is a polynomial operator corresponding to the slow diffusion regime. The main challenge stems from the structure of this nonlinear operator, which causes the operator to lose maximal monotonicity and prevents the use of standard weak-strong convergence techniques to identify the limit. Moreover, since the Laplace operator makes it difficult to separate the operators and discuss them individually, we attempt to obtain the of limit by investigating the global properties of the operator.
\end{abstract}
\noindent \textbf{MSCA 2020:} 60H15, 35R60, 35K65, 60H30.

\section{Introduction}

In the present work we are interested in the study of a backward stochastic (abbreviated to BSDE)
porous media equation of the form%
\begin{equation}
\left\{
\begin{array}{ll}
dY_{t}=-\left( \Delta \Psi \left( Y_{t}\right) +h\left( Y_{t},Z_{t}\right)
\right) dt+Z_{t}dW_{t}, & \quad \left( 0,T\right) \times \mathcal{O}\times
\Omega , \\
\Psi \left( Y_{t}\right) =0, & \quad \left( 0,T\right) \times \partial
\mathcal{O}\times \Omega , \\
Y_{T}=\xi , & \quad \mathcal{O},%
\end{array}%
\right.  \label{ecu1}
\end{equation}%
where $\mathcal{O}$ is a bounded open domain in $\mathbb{R}^{d}$ with smooth
boundary and $\left( \Omega ,\mathcal{F},\mathcal{F}_{t},\mathbb{P}\right) $
is a complete filtrated probability space, on which a cylindrical Wiener
process $\left\{ W_{t}\right\} _{t\geq 0}$ is defined on the separable
Hilbert space $H^{-1}\left( \mathcal{O}\right) $. As usually, we denoted by $%
\Delta $ the standard Laplace operator with homogeneous Dirichlet boundary conditions
and by $\Psi :\mathbb{R}\rightarrow \mathbb{R}$ a maximal monotone operator
with polynomial growth. The final data $\xi $ is assumed to be a $\mathcal{F}%
_{T}-$measurable random variables belonging to $L^{\infty }\left( \Omega ;L^{2}\left(
\mathcal{O}\right) \right) .$

Stochastic porous media equations extend classical nonlinear diffusion models by incorporating random perturbations, modeling phenomena such as fluid flow through heterogeneous porous materials under uncertainty. As a highly nonlinear partial differential equation, it also has an intrinsic mathematical interest which is not well posed in standard spaces of regular functions. These equations often involve maximal monotone operators and variational methods for existence and uniqueness, \cite{BDPR} systematically summarizes the relevant theory and provides a detailed exposition of the existence of solutions to the equation, their asymptotic behavior, and the ergodic properties of the associated transition semigroup. For some critical cases associated to applications including groundwater flow infiltration \cite{BII}, gravity driven infiltration further considered in \cite{eu-satur}, phases transition phenomena with convection terms induced by temperature gradients \cite{franco1}, problems with mushy regions and turbulence \cite{franco2}, and the asymptotic behaviors that are further studied in \cite{reika}. The growth order of $\Psi$ corresponds to problems with physically important problems, the porous medium equation describes phenomena ranging from slow diffusion with super-unit order, to fast diffusion with sub-unit order and  ultra-fast diffusion with negative order, even some models in complex media, e.g. \cite{shro}, \cite{robin}.

BSPDEs with porous media operators can model nonlinear diffusion with random fluctuations in heterogeneous materials, allowing the reconstruction of flow dynamics and source terms from later observations of the system. This capability is crucial in physics for inverse problems, where the aim is to infer past states or unknown inputs based on current measurements, particularly in complex porous media like groundwater or filtration systems.

Backward stochastic differential equations (BSDEs) within the finite-dimensional and ordinary setting, were first introduced in 1973 by Bismut as a tool for characterizing adjoint processes in the stochastic version of the Pontryagin maximum principle~\cite{Bismut:73}. The modern theory of nonlinear BSDEs, still referring to the finite-dimensional framework, was established by Pardoux and Peng~\cite{Pardoux/Peng:90}, who provided existence and uniqueness results for a broad class of equations.

Subsequently, Pardoux and Peng~\cite{Pardoux/Peng:92} initiated a stochastic approach to the existence theory for solutions of deterministic partial differential equations via BSDEs interpretations. Since then, BSDEs have undergone significant generalization, with much work devoted to approximation schemes and applications. BSDEs now serve as an essential framework in stochastic control, mathematical finance (notably in pricing derivatives), and the modeling of physical phenomena, as both linear and nonlinear BSDEs can describe pricing problems with or without portfolio constraints.

Pardoux and R\u{a}\c{s}canu~\cite{Pardoux/Rascanu:98} provided, via probabilistic representations, a tool adapted to the existence of viscosity solutions for multivalued parabolic and elliptic PDEs involving subdifferential operators. Their subsequent work~\cite{Pardoux/Rascanu:99} extended these results to equations in Hilbert spaces. Further extensions focussing on developing penalization techniques and convex analysis based on the Fitzpatrick function can be found in~\cite{Rascanu/Rotenstein:11}.

The variational approach to SPDEs, built on monotonicity and coercivity, was established by Krylov, Rozovskii~\cite{Krylov/Rozovskii:79} and Pardoux, Peng~\cite{Pardoux/Peng:92}, and further extended by Liu and R\"{o}ckner~\cite{Liu/Rockner:2010,Liu/Rockner:2013}.

More recently, backward stochastic partial differential equations (BSPDEs) have drawn substantial interest. Works like Tang, Wei~\cite{Tang/Wei:16}, Qiu, Tang, You~\cite{Qiu/Tang/You:2012}, and Sundar, Yin~\cite{Sundar/Yin:2009} examine BSPDEs and stochastic Navier-Stokes systems.

Finally, and without any aim of exhaustivity, we mention \cite{Liu}. The cited reference provides the
existence and uniqueness of strong solutions for a large class of BSPDEs,
by assuming that the coefficients satisfy a specific Lyapunov-type condition, instead of
the more classical coercivity one. The authors employ the generalized variational framework to substitute the standard monotonicity condition with a local monotonicity condition, thereby allowing applications to a wider range of quasilinear and semilinear BSPDE models.

Let us equally mention a few other directions of extension, without a particular link to our present study. Finite-dimensional BSDEs driven by discontinuous processes include fundamental contributions by Barles, Buckdahn, Pardoux~\cite{Barles/Buckdahn/Pardoux:97}, Tang, Li~\cite{Tang/Li:94}, Xia~\cite{Xia:00}, Carbone, Ferrario, Santacroce~\cite{Carbone/Ferrario/Santacroce:07}, and Cohen, Elliott~\cite{Cohen/Elliott:10}. Extensions to marked point processes and iterative methods for jump-driven equations can be found in works such as Confortola, Fuhrman~\cite{Confortola/Fuhrman:14}, Confortola, Fuhrman, Jacod~\cite{Confortola/Fuhrman/Jacod:16}, and Kharroubi, Lim~\cite{Kharroubi/Lim:14}, among others.

The main contribution of this work is the analysis of \emph{polynomial-growth} nonlinearities in the slow-diffusion regime, in contrast to \cite{Liu}, which concentrates on fast-diffusion equations. Methodologically, we rely on Yosida approximations rather than the Galerkin scheme adopted in \cite{Liu}.

From a technical standpoint, particular care is needed in handling the nonlinear coefficient $h$, which acts both in $H^{-1}$ (see the next sections for details) and in the larger space of the underlying Gelfand triple. In the polynomial-growth setting, identifying the limits of both the differential operator and the nonlinear term $h$ requires more delicate arguments. In particular, since the combined operator $\Delta\Psi + h$ is not maximally monotone, classical identification techniques break down, forcing us to develop new arguments. Through the Yosida approximation, we construct a family of approximating problems whose solutions are strong convergence in $H^{-1}$, while the associated combined operators are weakly convergence in $H^{-1}$. However, since dealing with the combined operator requires convergence results for $\Delta \Psi$ as well as further uniform estimates for the other coefficients, we try a second Yosida approximation for the equation. Then, by introducing a weighting function into the It\^{o}'s formula, we achieve the identification of the limiting nonlinear terms.

\medskip

\paragraph{Structure of the paper.}
The paper first motivates the model and reviews BSDE/BSPDE literature, then sets the functional-analytic framework and structural assumptions on \(\Psi\) and \(h\). Next, it defines the notion of solution and states the main theorem. A Yosida-approximated BSPDE is then solved using Liu's result -- only applying to Lipschitz settings, thus to our approximating equations,-- with a particular focus on uniform estimates for approximations. Finally, the passage to the limit and identification of the nonlinear terms are carried out, while preliminary technical lemmas concerning even further approximations are deferred to an Appendix.



\section{Preliminaries and assumptions}
In order to fix the setting, we recall the classical Gelfand triple employed for porous medium systems, i.e.
\begin{equation*}
L^{2}\left( \mathcal{O}\right) \subset H^{-1}\left( \mathcal{O}\right)
\subset \left( L^{2}\left( \mathcal{O}\right) \right) ^{\ast },
\end{equation*}%
see \cite{PrevotRoeckner2007}. In this context, we assume that the nonlinearity acts as follows
\begin{equation*}
h:L^{2}\left( \mathcal{O}\right) \times L_{2}\left( L^{2}\left( \mathcal{O}%
\right) ;H^{-1}\left( \mathcal{O}\right) \right) \rightarrow \left(
L^{2}\left( \mathcal{O}\right) \right) ^{\ast }.
\end{equation*}%
As usual, we have denoted by $L_{2}\left( L^{2}\left( \mathcal{O}\right)
;H^{-1}\left( \mathcal{O}\right) \right) $ the space of Hilbert-Schmidt
operators from $L^{2}\left( \mathcal{O}\right) $ to $H^{-1}\left( \mathcal{O}%
\right) $.\\
Furthermore, we consider the eigen-functions associated to $\Delta$ (the system being considered orthonormal w.r.t. $H^{-1}(\mathcal{O})$), denoted by $\pr{e_n}_{n\geq 1}$. We consider the finite-dimensional eigen-space
\[H_n:=span \set{e_1,\ldots,e_n}\subset H^{-1}(\mathcal{O}),\ for \ n\geq 1.\]

Throughout the paper, we employ the following set of assumptions.\\

\noindent \textbf{Assumptions}

\begin{itemize}
\item[i)] The function $\Psi:\mathbb{R}\longrightarrow\mathbb{R}$ is assumed to be
\begin{enumerate}
    \item[a)] monotonically increasing,
\item[b)] continuously differentiable, and
\item[c)] such that, for some $m\geq 1$,%
\begin{equation*}
\left\{
\begin{array}{l}
\Psi \left( 0\right) =0,\quad \alpha _{0}<\Psi ^{\prime }\left( r\right)
\leq \alpha _{1}\left\vert r\right\vert ^{m-1}+\alpha _{2}, \\
\Psi(r)r\geq \int_{0}^{r}\Psi \left( s\right) ds\geq \alpha _{3}\left\vert r\right\vert
^{m+1}+\alpha _{4}\left\vert r\right\vert
^{2},
\end{array}%
\right.
\end{equation*}%
for all $r\in \mathbb{R}$, where $\alpha _{0},~\alpha _{3}>0,~\alpha _{1},\ \alpha_2 \geq 0$, $\alpha _{4}\in\mathbb{R}$,
are fixed constants.
\end{enumerate}
\item[ii)] The nonlinearity $h$ is assumed to satisfy

\begin{itemize}
\item[a)] (\textit{hemicontinuity}) For all $y_{1},~y_{2}\in L^{2}\left(
\mathcal{O}\right) $, all $z\in L_{2}\left( L^{2}\left( \mathcal{O}\right)
;H^{-1}\left( \mathcal{O}\right) \right) $ and all $\phi \in L^{2}\left(
\mathcal{O}\right) $ the application%
\begin{equation*}
\mathbb{R}\ni s\mapsto _{\left( L^{2}\left( \mathcal{O}\right) \right) ^{\ast
}}\left\langle h\left( y_{1}+sy_{2},z\right) ,\phi \right\rangle
_{L^{2}\left( \mathcal{O}\right) }
\end{equation*}%
is continuous.

\item[b)] (\textit{monotonicity in the first variable}) {There exists a locally
bounded measurable function $\rho :L^{2}\left( \mathcal{O}\times \Omega\times [0,T]\right)
\rightarrow \mathbb{R}_{+}$ such that}
\begin{equation*}
\begin{aligned}
&_{\left( L^{2}\left( \mathcal{O}\right) \right) ^{\ast }}\left\langle
h\left( y_{1},z_{1}\right) -h\left( y_{2},z_{2}\right)
,y_{1}-y_{2}\right\rangle _{_{L^{2}\left( \mathcal{O}\right) }} \\
\leq&\rho \left( y_{1}\right) \left[\left\Vert y_{1}-y_{2}\right\Vert
_{H^{-1}\left( \mathcal{O}\right) }^{2} +\left\Vert y_{1}-y_{2}\right\Vert _{H^{-1}\left(
\mathcal{O}\right) }\left\Vert z_{1}-z_{2}\right\Vert _{L_{2}\left(
L^{2}\left( \mathcal{O}\right) ;H^{-1}\left( \mathcal{O}\right) \right) }\right],
\end{aligned}
\end{equation*}
for all $y_{1},~y_{2}\in L^{2}\left( \mathcal{O}\right) $, and all $%
z_{1},\ z_{2}\in L_{2}\left( L^{2}\left( \mathcal{O}\right) ;H^{-1}\left(
\mathcal{O}\right) \right) .$

\item[c)] (\textit{one-sided linear growth})
There exist generic constants $K,\ \gamma$ such that the following hold true.
\begin{itemize}
\item[$c_{1})$] (\emph{In $L^{2}\left( \mathcal{O}\right) $}) For every $n\in \mathbb{N}$,
the operator $h$ maps $H_n$ into $L^{2}\left( \mathcal{O%
}\right) $, and, for every $y\in H_{n}$ and $z\in L_{2}\left( L^{2}\left( \mathcal{O}\right)
;H^{-1}\left( \mathcal{O}\right) \right) $
\begin{equation*}
\left\langle h\left( y,z\right) ,y\right\rangle _{L^{2}\left( \mathcal{O}%
\right) }\leq K\left\Vert y\right\Vert _{L^{2}\left( \mathcal{O}\right)
}^{2}+\gamma \left\Vert z\right\Vert _{L_{2}\left( L^{2}\left( \mathcal{O}%
\right) ;L^{2}\left( \mathcal{O}\right) \right) }^{2};
\end{equation*}

\item[$c_{2})$] (\emph{In $H^{-1}\left( \mathcal{O}\right) $})
\begin{equation*}
\left\langle h\left( y,z\right) ,y\right\rangle _{H^{-1}\left( \mathcal{O}%
\right) }\leq K\left\Vert y\right\Vert _{H^{-1}\left( \mathcal{O}\right)
}^{2}+\gamma \left\Vert z\right\Vert _{L_{2}\left( L^{2}\left( \mathcal{O}%
\right) ;H^{-1}\left( \mathcal{O}\right) \right) }^{2},
\end{equation*}
for all $y\in H^{-1}\left( \mathcal{O}\right) $, and all $z\in L_{2}\left(
L^{2}\left( \mathcal{O}\right) ;H^{-1}\left( \mathcal{O}\right) \right) .$
\end{itemize}

\item[d)] (\textit{growth}) There exists a positive constant $K$ such that\footnote{with a function $\rho$ behaving in the same way as the one appearing in the monotonicity assumption b).}%
\begin{itemize}
\item[$d_{1})$]
\begin{equation*}
\left\Vert h\left( y,z\right) \right\Vert _{\left( L^{2}\left( \mathcal{O}%
\right) \right) ^{\ast }}\leq K\left( \rho \left( y\right) +\left\Vert
z\right\Vert _{L_{2}\left( L^{2}\left( \mathcal{O}\right) ;L^{2}\left(
\mathcal{O}\right) \right) }\right) ,
\end{equation*}%
for all $y\in L^{2}\left( \mathcal{O}\right) $, and all $z\in L_{2}\left(
L^{2}\left( \mathcal{O}\right) ;L^{2}\left( \mathcal{O}\right) \right) $;

\item[$d_{2})$]and
 \begin{equation*}
\left\Vert h\left( y,z\right) \right\Vert _{H^{-1}\left( \mathcal{O}\right)
}\leq K\left( \rho \left( y\right) +\left\Vert z\right\Vert _{L_{2}\left(
L^{2}\left( \mathcal{O}\right) ;H^{-1}\left( \mathcal{O}\right) \right)
}\right),
\end{equation*}%
for all $y\in L^{2}\left( \mathcal{O}\right) $ and all $z\in L_{2}\left(
L^{2}\left( \mathcal{O}\right) ;H^{-1}\left( \mathcal{O}\right) \right) $.
\end{itemize}
\end{itemize}
\end{itemize}

\begin{remark}\begin{enumerate}
\item The assertions $c_{1})$ and $d_{1})$ are necessary to apply the existence result from \cite{Liu} in the Gelfand triple $L^{2}\left( \mathcal{O}\right) \subset H^{-1}\left( \mathcal{O}\right)
\subset \left( L^{2}\left( \mathcal{O}\right) \right) ^{\ast }$ for the Lipschitz approximating equation. \item The present work treats the general porous media equation in $H^{-1}\left( \mathcal{O}\right)$; hypotheses $c_{2})$ and $d_{2})$ intervene in most of the estimates with respect to $H^{-1}\left( \mathcal{O}\right)$.\end{enumerate}
\end{remark}

\noindent \textbf{Notations}
\begin{itemize}
\item Considering the complete filtered probability space $\left( \Omega ,\mathcal{F},\mathcal{F}_{t},\mathbb{P}\right) $, we denote by $\mathcal{P}$\ the $\sigma -$algebra of the predictable sets on $\Omega \times \left[ 0,T\right]
$ associated with $\left\{ \mathcal{F}_{t}\right\} _{t\in \left[ 0,T\right]
} $.
\item For some generic Banach space $B$, we let $L_{\mathcal{F}}^{p}\left( \Omega
;L^{r}\left( \left[ 0,T\right] ;B\right) \right) $ with $p,\ r\in \left[
1,\infty \right] $ be the set of all predictable $B$-valued processes in $%
L^{p}\left( \Omega ;L^{r}\left( \left[ 0,T\right] ;B\right) \right) $ and
let $L_{\mathcal{F}}^{p}\left( \Omega \times \left[ 0,T\right] ;B\right) =L^{p}_{\mathcal{F}}\left( \Omega ;L^{p}\left( \left[ 0,T\right] ;B\right) \right) $.
\item We denote by $L_{\mathcal{F}%
_{T}}^{p}\left( \Omega ;B\right)$,\ $p\in \left[ 1,\infty \right] $ the set
of all $\mathcal{F}_{T}-$measurable random variables in $L^{p}\left( \Omega
;B\right) $.\end{itemize}

\section{Definition of the solution and main result}

\begin{definition}\label{Def1}
For a terminal condition $\zeta \in L_{\mathcal{F}_{T}}^{\infty }\left( \Omega ;L^{2}\left(
\mathcal{O}\right) \right) $ we call a solution to (\ref{ecu1}) some pair $%
\left( Y,Z\right) $ which satisfies
\begin{equation*}
Y\in L_{\mathcal{F}}^{2}\left( \Omega ;L^{\infty }\left( \left[ 0,T\right]
;L^{2}\left( \mathcal{O}\right) \right) \right) \cap L_{\mathcal{F}%
}^{2}\left( \Omega ;C\left( \left[ 0,T\right] ;H^{-1}\left( \mathcal{O}%
\right) \right) \right);
\end{equation*}%
\begin{equation*}
Z\in L_{\mathcal{F}}^{2}\left( \Omega \times \left[ 0,T\right]
;L_{2}\left( L^{2}\left( \mathcal{O}\right) ;H^{-1}\left( \mathcal{O}\right)
\right) \right);
\end{equation*}%
and%
\begin{equation*}
\begin{aligned}
\left\langle Y_{t},e_{k}\right\rangle _{H^{-1}\left( \mathcal{O}\right) }
=&\left\langle \zeta ,e_{k}\right\rangle _{H^{-1}\left( \mathcal{O}\right)
}-\int_{t}^{T}\int_{\mathcal{O}}\Psi \left( Y_{s}(\xi)\right) e_{k}(\xi)d\xi ds \\
&+\int_{t}^{T}\left\langle h\left( Y_{s},Z_{s}\right) ,e_{k}\right\rangle
_{H^{-1}\left( \mathcal{O}\right) }ds-\int_{t}^{T}\left\langle
e_{k},Z_{s}dW_{s}\right\rangle _{H^{-1}\left( \mathcal{O}\right) },
\end{aligned}
\end{equation*}
$\mathbb{P}$-a.s. for every eigen-function $e_{k}$ as introduced before.
\end{definition}

We formulate the main consistency result.

\begin{theorem}\label{Th1}
Under the assumptions above, for each $\zeta \in L_{\mathcal{F}_{T}}^{\infty
}\left( \Omega ;L^{2}\left( \mathcal{O}\right) \right) $ the backward
stochastic partial differential equation (\ref{ecu1}) has a unique adapted
solution $\left( Y,Z\right) $ in the sense of Definition \ref{Def1}.
\end{theorem}

\begin{proof}
The first step is to construct a Yosida approximation of the operator $\Psi$ to handle its polynomial growth, thereby replacing it with an approximation involving Lipschitz continuous operators.

More precisely, we consider, for each $\lambda >0$ the approximation%
\begin{equation*}
\Psi _{\lambda }\left( x\right) =\frac{1}{\lambda }\left( x-\left(
Id+\lambda \Psi \right) ^{-1}\left( x\right) \right) =\Psi \left( \left(
Id+\lambda \Psi \right) ^{-1}\left( x\right) \right),
\end{equation*}%
where $Id$ is the identity function on $\mathbb{R}$. We further consider the corresponding approximating equation%
\begin{equation}
\left\{
\begin{array}{ll}
dY_{t}^{\lambda }=-\left( \Delta \Psi _{\lambda }\left( Y_{t}^{\lambda
}\right) +h\left( Y_{t}^{\lambda },Z_{t}^{\lambda }\right) \right)
dt+Z_{t}^{\lambda }dW_{t}, & \quad \left( 0,T\right) \times \mathcal{O}%
\times \Omega , \\
\Psi _{\lambda }\left( Y_{t}^{\lambda }\right) =0, & \quad \left( 0,T\right)
\times \partial \mathcal{O}\times \Omega , \\
Y_{T}^{\lambda }=\xi , & \quad \mathcal{O}.%
\end{array}%
\right.  \label{aprox1}
\end{equation}

The reader is reminded that $\Psi _{\lambda }$ inherits the monotonicity properties form $%
\Psi $. In addition, it exhibits linear growth of type $C_{1}\left( 1+\left\vert x\right\vert \right)
$ for all $x\in \mathbb{R}$, albeit the fact that the constant is not uniform in $\lambda>0$.

We are now able to invoke the consistency result in \cite[Theorem 2.2]{Liu} in order to get that the equation (\ref%
{aprox1}) has a unique solution in the sense of Definition \ref{Def1}, solution which
further satisfies
\begin{equation}
\mathbb{E}\left[\underset{t\in \left[ 0,T\right] }{\sup }\left\Vert Y_{t}^{\lambda
}\right\Vert _{L^{2}\left( \mathcal{O}\right) }^{2}\right]+\frac{1}{2}\mathbb{E}\left[
\int_{0}^{T}\left\Vert Z_{s}^{\lambda }\right\Vert _{L_{2}\left( L^{2}\left(
\mathcal{O}\right) ;L^{2}\left( \mathcal{O}\right) \right) }^{2}ds\right]\leq
C\left\Vert \xi \right\Vert _{L_{\mathcal{F}_{T}}^{\infty }\left( \Omega
;L^{2}\left( \mathcal{O}\right) \right) }^{2},\quad a.s.  \label{estim-vechi}
\end{equation}
\bigskip

\textbf{Part 1 (weak convergence in H}$^{-1}$\textbf{)}

\bigskip

Keeping in mind the assumption $c)$ and the estimate (\ref{estim-vechi}) we
can apply the It\^{o} formula in $H^{-1}\left( \mathcal{O}\right) $ to the
equation%
\begin{equation}
Y_{t}^{\lambda }=\xi +\int_{t}^{T}\left( \Delta \left( \Psi _{\lambda
}\right) \left( Y_{s}^{\lambda }\right) +h\left( Y_{s}^{\lambda
},Z_{s}^{\lambda }\right) \right) ds-\int_{t}^{T}Z_{s}^{\lambda }dW_{s},
\label{sol-aprox}
\end{equation}
and we get
\begin{equation}\label{Ito1}
\begin{aligned}
&e^{\alpha t}\left\Vert Y_{t}^{\lambda }\right\Vert _{H^{-1}\left( \mathcal{%
O}\right) }^{2}+\alpha \int_{t}^{T}e^{\alpha s}\left\Vert Y_{s}^{\lambda }\right\Vert
_{H^{-1}\left( \mathcal{O}\right) }ds+2\int_{t}^{T}e^{\alpha s}\left\langle \Psi _{\lambda }\left(
Y_{s}^{\lambda }\right) ,Y_{s}^{\lambda }\right\rangle _{L^{2}\left(
\mathcal{O}\right) }ds   \\
=&e^{\alpha T}\left\Vert \xi \right\Vert _{H^{-1}\left( \mathcal{O}\right)
}^{2}+2\int_{t}^{T}e^{\alpha s}\left\langle h\left( Y_{s}^{\lambda
},Z_{s}^{\lambda }\right) ,Y_{s}^{\lambda }\right\rangle _{H^{-1}\left(
\mathcal{O}\right) }ds  \\
&-2\int_{t}^{T}e^{\alpha s}\left\langle Y_{s}^{\lambda },Z_{s}^{\lambda
}dW_{s}\right\rangle _{H^{-1}\left( \mathcal{O}\right)
}-\int_{t}^{T}e^{\alpha s}\left\Vert Z_{s}^{\lambda }\right\Vert
_{L_{2}\left( L^{2}\left( \mathcal{O}\right) ,H^{-1}\left( \mathcal{O}
\right) \right) }^{2}ds.
\end{aligned}
\end{equation}

We write%

\begin{equation*}
\begin{aligned}
&\left\langle \Psi _{\lambda }\left( Y_{s}^{\lambda}\right)
,Y_{s}^{\lambda }\right\rangle _{L^{2}\left( \mathcal{O}\right) } \\
=&\left\langle \Psi \left( J_{\lambda }\left( Y_{s}^{\lambda }\right)
\right) ,J_{\lambda }\left( Y_{s}^{\lambda }\right) -J_{\lambda }\left(
Y_{s}^{\lambda }\right) +Y_{s}^{\lambda }\right\rangle _{L^{2}\left(
\mathcal{O}\right) } \\
=&\left\langle \Psi \left( J_{\lambda }\left( Y_{s}^{\lambda }\right)
\right) ,J_{\lambda }\left( Y_{s}^{\lambda }\right) \right\rangle
_{L^{2}\left( \mathcal{O}\right) }+\lambda \left\Vert Y_{s}^{\lambda
}-J_{\lambda }\left( Y_{s}^{\lambda }\right) \right\Vert _{L^{2}\left(
\mathcal{O}\right) }^{2},
\end{aligned}
\end{equation*}
and%

\begin{equation*}
\left\langle h\left( Y_{s}^{\lambda },Z_{s}^{\lambda }\right)
,Y_{s}^{\lambda }\right\rangle _{H^{-1}\left( \mathcal{O}\right) }\leq
K\left\Vert Y_{s}^{\lambda }\right\Vert _{H^{-1}\left( \mathcal{O}\right)
}^{2}+\gamma \left\Vert Z_{s}^{\lambda }\right\Vert _{L_{2}\left(
L^{2}\left( \mathcal{O}\right) ;H^{-1}\left( \mathcal{O}\right) \right)
}^{2}.
\end{equation*}

Going back to estimate (\ref{Ito1}) we get that
\begin{equation}\label{Ito7}
\begin{aligned}
&\left\Vert Y_{t}^{\lambda }\right\Vert _{H^{-1}\left( \mathcal{O}\right)
}^{2}+\int_{t}^{T}e^{\alpha \left( s-t\right) }\left\langle \Psi \left(
J_{\lambda }\left( Y_{s}^{\lambda }\right) \right) ,J_{\lambda }\left(
Y_{s}^{\lambda }\right) \right\rangle _{L^{2}\left( \mathcal{O}\right) }ds \\
&+\lambda \int_{t}^{T}e^{\alpha \left( s-t\right) }\left\Vert
Y_{s}^{\lambda }-J_{\lambda }\left( Y_{s}^{\lambda }\right) \right\Vert
_{L^{2}\left( \mathcal{O}\right) }^{2}ds \\
\leq& e^{\alpha \left( T-t\right) }\left\Vert \xi \right\Vert
_{H^{-1}\left( \mathcal{O}\right) }^{2}+\left( 2K-\alpha \right)
\int_{t}^{T}e^{\alpha \left( s-t\right) }\left\Vert Y_{s}^{\lambda
}\right\Vert _{H^{-1}\left( \mathcal{O}\right) }^{2}ds   \\
&-2\int_{t}^{T}e^{\alpha \left( s-t\right) }\left\langle Y_{s}^{\lambda
},Z_{s}^{\lambda }dW_{s}\right\rangle _{H^{-1}\left( \mathcal{O}\right) }   \\
&+\left( 2\gamma -1\right) \int_{t}^{T}e^{\alpha \left( s-t\right)
}\left\Vert Z_{s}^{\lambda }\right\Vert _{L_{2}\left( L^{2}\left( \mathcal{O}
\right) ,H^{-1}\left( \mathcal{O}\right) \right) }^{2}ds.
\end{aligned}
\end{equation}
Keeping in mind that
\begin{equation*}
   \left\langle \Psi \left( J_{\lambda }\left( Y_{s}^{\lambda }\right)
\right) ,J_{\lambda }\left( Y_{s}^{\lambda }\right) \right\rangle
_{L^{2}\left( \mathcal{O}\right) } \geq \alpha _3\left\vert J_{\lambda }\left( Y_{s}^{\lambda
}\right) \right\vert ^{m+1}+\alpha _{4}\left\vert J_{\lambda }\left( Y_{s}^{\lambda
}\right)\right\vert
^{2},
\end{equation*}
and taking the expectation in the previous relation we obtain, for $\alpha >2K$ and $2\gamma <1$, that
\begin{equation}\label{Ito8}
\begin{aligned}
&\mathbb{E}\left[\left\Vert Y_{t}^{\lambda }\right\Vert _{H^{-1}\left( \mathcal{O}
\right) }^{2}\right]+\mathbb{E}\left[\int_{t}^{T}e^{\alpha \left( s-t\right) } \left( \alpha _{3}\left\vert
J_{\lambda }\left( Y_{s}^{\lambda }\right) \right\vert ^{m+1}+\alpha _{4}\left\vert J_{\lambda }\left( Y_{s}^{\lambda
}\right)\right\vert
^{2} \right) ds\right]  \\
&+\lambda \mathbb{E}\left[ \int_{t}^{T}e^{\alpha \left( s-t\right) }\left\Vert
Y_{s}^{\lambda }-J_{\lambda }\left( Y_{s}^{\lambda }\right) \right\Vert
_{L^{2}\left( \mathcal{O}\right) }^{2}ds\right]+\left( \alpha -2K\right) \mathbb{E}\left[\int_{t}^{T}e^{\alpha \left(
s-t\right) }\left\Vert Y_{s}^{\lambda }\right\Vert _{H^{-1}\left( \mathcal{O}%
\right) }^{2}ds\right]\\
&+\left( 1-2\gamma \right) \mathbb{E}\left[\int_{t}^{T}e^{\alpha \left(
s-t\right) }\left\Vert Z_{s}^{\lambda }\right\Vert _{L_{2}\left( L^{2}\left(
\mathcal{O}\right) ,H^{-1}\left( \mathcal{O}\right) \right) }^{2}ds\right]\\
\leq& C\left[1+\mathbb{E}e^{\alpha \left( T-t\right) }\left\Vert \xi \right\Vert
_{H^{-1}\left( \mathcal{O}\right) }^{2}\right],
\end{aligned}
\end{equation}
for some constant $C$ independent of $\lambda $ (but depending on $\alpha_3,\ \alpha_4$ and so on; this constant is generic and is allowed to change from one line of the proof to another, in order to facilitate reading).

We take now the supremum on $t$, followed by the expectation in (\ref{Ito7}) and we get
\begin{equation}\label{BDG7}
\begin{aligned}
& \mathbb{E}\left[ \underset{\tau \in \left[ t,T\right] }{\sup }e^{\alpha \tau}\left\Vert
Y_{\tau }^{\lambda}\right\Vert _{H^{-1}\left( \mathcal{O}
\right) }^{2}\right] +\left( 1-2\gamma\right) \mathbb{E} \left[\int_{t}^{T}e^{\alpha s
}\left\Vert Z_{s}^{\lambda }\right\Vert _{L_{2}\left( L^{2}\left( \mathcal{O}
\right) ,H^{-1}\left( \mathcal{O}\right) \right) }^{2}ds \right] \\
\leq& e^{\alpha T}\mathbb{E}\left[\left\Vert \xi \right\Vert
_{H^{-1}\left( \mathcal{O}\right) }^{2}\right]+2\mathbb{E}\left[ \underset{\tau \in \left[ t,T\right] }{\sup } \left\vert\int_{\tau}^{T}e^{\alpha s }\left\langle Y_{s}^{\lambda
},Z_{s}^{\lambda }dW_{s}\right\rangle _{H^{-1}\left( \mathcal{O}\right) } \right\vert \right].
\end{aligned}
\end{equation}
By the Burkholder-Davis-Gundy inequality we have that%
\begin{equation*}
\begin{aligned}
&\mathbb{E}\left[ \underset{\tau \in \left[ t,T\right] }{\sup }\left\vert
\int_{\tau }^{T}e^{\alpha s}\left\langle Y_{s}^{\lambda },Z_{s}^{\lambda
}dW_{s}\right\rangle _{H^{-1}\left( \mathcal{O}\right)}\right\vert \right]  \\
\leq& C\mathbb{E}\left[ \int_{t}^{T}e^{\alpha s}\left\Vert Y_{s}^{\lambda
}\right\Vert _{H^{-1}\left( \mathcal{O}\right) }^{2}\left\Vert Z_{s}^{\lambda
}\right\Vert _{L_{2}\left( L^{2}\left( \mathcal{O}\right)
,H^{-1}\left( \mathcal{O}\right) \right) }^{2}ds\right] ^{1/2} \\
\leq& C\left\{ \frac{1}{4C}\mathbb{E}\left[\underset{\tau \in \left[ t,T\right] }{%
\sup }e^{\alpha \tau}\left\Vert Y_{\tau }^{\lambda}\right\Vert _{H^{-1}\left( \mathcal{O}\right) }^{2}\right]+C\mathbb{E}\left[\int_{t}^{T}e^{\alpha s}\left\Vert
Z_{s}^{\lambda}\right\Vert _{L_{2}\left( L^{2}\left( \mathcal{O%
}\right) ,H^{-1}\left( \mathcal{O}\right) \right) }^{2}ds\right]\right\},
\end{aligned}
\end{equation*}
which, by going back to (\ref{BDG7}), and recalling the bounds on the $Z$ term from (\ref{Ito8}), leads to
\begin{eqnarray*}
&&\mathbb{E}\left[ \underset{\tau \in \left[ t,T\right] }{\sup }%
e^{\alpha \tau}\left\Vert Y_{\tau }^{\lambda}\right\Vert _{H^{-1}\left( \mathcal{O}
\right)}^{2}\right] \leq C \left[ \mathbb{E}\left\Vert \xi \right\Vert _{H^{-1}\left( \mathcal{O}
\right)}^{2}+1\right].
\end{eqnarray*}

We obtained that
\begin{eqnarray*}
&&\left\{ J_{\lambda }\left( Y^{\lambda }\right) \right\} _{\lambda }%
\text{ is bounded in }L^{m+1}\left( \Omega \times \left( 0,T\right) \times
\mathcal{O}\right); \\
&&\left\{ Y^{\lambda }\right\} _{\lambda }\text{is bounded in }L^{\infty
}\left( 0,T;L^{2}\left( \Omega ;H^{-1}\left( \mathcal{O}\right) \right)
\right) \cap L^{2}\left( \Omega ;C\left( \left[ 0,T\right] ;H^{-1}\left(
\mathcal{O}\right) \right) \right); \\
&&\left\{ Z^{\lambda }\right\} _{\lambda }\text{ is bounded in }%
L^{2}\left( \left( 0,T\right) \times \Omega ;L_{2}\left( L^{2}\left(
\mathcal{O}\right) ,H^{-1}\left( \mathcal{O}\right) \right) \right) .
\end{eqnarray*}

Owing to the assumptions on $\Psi $, one has
\begin{equation*}
\left\vert \Psi _{\lambda }\left( Y_{s}^{\lambda }\right) \right\vert ^{%
\frac{m+1}{m}}=\left\vert \Psi \left( J_{\lambda }\left( Y_{s}^{\lambda
}\right) \right) \right\vert ^{\frac{m+1}{m}}\leq C(\left\vert J_{\lambda
}\left( Y_{s}^{\lambda }\right) \right\vert ^{m+1}+\vert J_{\lambda
}\left( Y_{s}^{\lambda }\right) \vert ^{\frac{m+1}{m}}+1).
\end{equation*}%
As a consequence,
\begin{equation}
\left\{ \Psi _{\lambda }\left( Y^{\lambda }\right) \right\} _{\lambda }%
\text{ is bounded in }L^{\frac{m+1}{m}}\left( \Omega \times \left(
0,T\right) \times \mathcal{O}\right) .  \label{marg-psi}
\end{equation}

From the assumptions on $h$ we have that
\begin{equation*}
\mathbb{E}\left[\left\Vert h\left( Y_{s}^{\lambda },Z_{s}^{\lambda }\right) \right\Vert
_{\left( L^{2}\left( \mathcal{O}\right) \right) ^{\ast }}^2\right]\leq K\left( \rho
\left( Y_{s}^{\lambda }\right) +\left\Vert Z_{s}^{\lambda }\right\Vert
_{L_{2}\left( L^{2}\left( \mathcal{O}\right) ;H^{-1}\left( \mathcal{O}\right)
\right) }\right) \leq C.
\end{equation*}

We have for the moment the following weak convergences%
\begin{eqnarray*}
J_{\lambda }\left( Y_{s}^{\lambda }\right) &\rightharpoonup &Y_{s}\text{ in }%
L^{m+1}\left( \Omega \times \left( 0,T\right) \times \mathcal{O}\right); \\
Z_{s}^{\lambda } &\rightharpoonup &Z_{s}\text{ in }L^{2}\left( \left(
0,T\right) \times \Omega ;L_{2}\left( L^{2}\left( \mathcal{O}\right)
,H^{-1}\left( \mathcal{O}\right) \right) \right); \\
\Psi _{\lambda }\left( Y_{s}^{\lambda }\right) &\rightharpoonup &\eta \text{
in }L^{\frac{m+1}{m}}\left( \Omega \times \left( 0,T\right) \times \mathcal{O%
}\right); \\
h\left( Y_{s}^{\lambda },Z_{s}^{\lambda }\right) &\rightharpoonup &\phi
\text{ in }L^{2}\left( \left( 0,T\right) \times \Omega ;\left( L^{2}\left(
\mathcal{O}\right) \right) ^{\ast }\right) .
\end{eqnarray*}%
We can now pass to the limit and get%
\begin{equation*}
\begin{aligned}
\left\langle Y_{t},e_{k}\right\rangle _{H^{-1}\left( \mathcal{O}\right) }
=&\left\langle \xi ,e_{k}\right\rangle _{H^{-1}\left( \mathcal{O}\right) }
\\
&-\int_{t}^{T}\int_{\mathcal{O}}\eta e_{k}d\xi ds+\int_{t}^{T}\left\langle
\phi ,e_{k}\right\rangle _{H^{-1}\left( \mathcal{O}\right) }ds \\
&-\int_{t}^{T}\left\langle e_{k},Z_{s}dW_{s}\right\rangle _{H^{-1}\left(
\mathcal{O}\right) },~\mathbb{P}-a.s.
\end{aligned}
\end{equation*}

In order to conclude the proof we still need to identify $\Psi \left(Y\right) =\eta $ and $h\left( Y,Z\right)
=\phi $ and for this we need some new estimates.

\bigskip

\textbf{Part 2 (identification of the limit)}

\bigskip

This section is the most technical part of the paper and needs to be treated
in several steps.

The main difficulty comes from the fact that the nonlinear operator $\Delta
\Psi +h$ is not maximal monotone in the natural space $L^{2}\left(
0,T;L^{2}\left( \Omega ;H^{-1}\left( \mathcal{O}\right) \right) \right) $
and therefore we can't apply directly the classical results. More precisely,
we need to prove first some preliminary estimates and finally to deduce the
identification of the limit by a new result which is adapted to our
situation.

In order to prove this identification, we need to use some monotonicity
arguments.

For this reason, we note that the operator
\begin{equation*}
\begin{aligned}
A_{\lambda } :\ &D\left( A_{\lambda }\right) \subset H^{-1}\left( \mathcal{O}%
\right) \rightarrow H^{-1}\left( \mathcal{O}\right); \\
A_{\lambda }\left( y\right) =&-\Delta \Psi _{\lambda }\left( y\right); \\
D\left( A_{\lambda }\right) =&\left\{ y\in H^{-1}\left( \mathcal{O}\right)
\cap L^{1}\left( \mathcal{O}\right) :\Psi _{\lambda }\left( x\right) \in
H_{0}^{1}\left( \mathcal{O}\right) \right\} ,
\end{aligned}
\end{equation*}%
is maximal monotone in $L^{2}\left( \mathcal{O}\right) \subset H^{-1}\left(
\mathcal{O}\right) \subset \left( L^{2}\left( \mathcal{O}\right) \right)
^{\ast }$ and, on the other hand, the operator
\begin{equation*}
h:L^{2}\left( \mathcal{O}\right) \times L_{2}\left( L^{2}\left( \mathcal{O}%
\right) ;H^{-1}\left( \mathcal{O}\right) \right) \rightarrow \left(
L^{2}\left( \mathcal{O}\right) \right) ^{\ast },
\end{equation*}%
is monotone and continuous in the first variable.

Since the two nonlinear operators can not be treated separately, we need
first to denote%
\begin{equation*}
\begin{aligned}
U_{\lambda } &:L^{2}\left( \mathcal{O}\right) \times L_{2}\left(
L^{2}\left( \mathcal{O}\right) ;H^{-1}\left( \mathcal{O}\right) \right)
\rightarrow \left( L^{2}\left( \mathcal{O}\right) \right) ^{\ast } ,\\
U_{\lambda }\left( y,z\right) &=-A_{\lambda }\left( y\right) +h\left(
y,z\right) .
\end{aligned}
\end{equation*}
Since
\begin{equation*}
-\int_{t}^{T}\int_{\mathcal{O}}\eta e_{k}d\xi ds+\int_{t}^{T}\left\langle
\phi ,e_{k}\right\rangle _{H^{-1}\left( \mathcal{O}\right)
}ds=\int_{t}^{T}\left\langle \Delta \eta +\phi ,e_{k}\right\rangle
_{H^{-1}\left( \mathcal{O}\right) }ds,
\end{equation*}%
and we need to show that%
\begin{equation*}
 \Delta
\eta +\phi=\Delta \Psi \left( Y\right) +h\left( Y,Z\right),
\end{equation*}
with $\Delta \Psi \left( Y\right) +h\left( Y,Z\right) = U\left( Y,Z\right)$.

\bigskip

\textbf{Step I (weak convergences \textbf{of }}$\mathbf{U}_{\lambda }\left(
\mathbf{Y_{t}^{\lambda },Z_{t}^{\lambda }}\right) $\textbf{\textbf{\ }in }$%
H^{-1}$\textbf{)}

\bigskip

In this step we intend to prove that
\begin{equation}
\mathbb{E}\int_{t}^{T}\left\Vert U_{\lambda }\left( Y_{s}^{\lambda
},Z_{s}^{\lambda }\right) \right\Vert _{-1}^{2}ds<C.  \label{margU}
\end{equation}

In order to study the operator $U_{\lambda }$ in $H^{-1}\left( \mathcal{O}%
\right) $\ which means that the operator $\Psi _{\lambda }\left(
Y_{s}^{\lambda }\right) $ in $H_{0}^{1}\left( \mathcal{O}\right) $ we need
to take a second approximation in $\varepsilon $ and to pass to the limit as
follows.

Since the operators $A_{\lambda }$ is maximal monotone in $H^{-1}\left(
\mathcal{O}\right) $ we can take the Yosida approximation in $H^{-1}\left(
\mathcal{O}\right) $ which has the form%
\begin{equation*}
\left( A_{\lambda }\right) ^{\varepsilon }=\frac{1}{\varepsilon }\left(
Id-\left( Id+\varepsilon A_{\lambda }\right) ^{-1}\right) ,\quad \forall
\varepsilon >0,
\end{equation*}%
where $Id$ is the identity in $H^{-1}\left( \mathcal{O}\right) $.

We denote the resolvent of $A_{\lambda }$ by $J_{\varepsilon }\left(
x\right) =\left( Id+\varepsilon A_{\lambda }\right) ^{-1}\left( x\right) .$

We take now a second approximation for the equation which is
\begin{equation}
\left\{
\begin{array}{ll}
dY_{t}^{\lambda ,\varepsilon }=\left( A_{\lambda }\right) ^{\varepsilon
}\left( Y_{t}^{\lambda ,\varepsilon }\right) dt-h\left( Y_{t}^{\lambda
,\varepsilon },Z_{t}^{\lambda ,\varepsilon }\right) +Z_{t}^{\lambda
,\varepsilon }dW_{t}, & \left( 0,T\right) \times \mathcal{O}\times \Omega ,
\\
\Psi _{\lambda }\left( Y_{t}^{\lambda ,\varepsilon }\right) =0, & \left(
0,T\right) \times \partial \mathcal{O}\times \Omega , \\
Y_{T}^{\lambda ,\varepsilon }=\xi , & \mathcal{O}.%
\end{array}%
\right.  \label{aprox2}
\end{equation}

Note that the operator $\left( A_{\lambda }\right) ^{\varepsilon }$ is
maximal monotone and Lipschitz in $H^{-1}\left( \mathcal{O}\right) $ and
also $L^{2}\left( \mathcal{O}\right) .$

We can start now to prepare the convergences in $\varepsilon $ through the
following estimates.

The detailed proofs follow fairly standard arguments and are therefore collected in the Appendix (Lemmas \ref{lema1}, \ref{lema2}, \ref{lema3}), both to keep the paper self-contained and to improve the readability of the main text.

\begin{lemma}
\label{lema1}Under the assumptions above, we have
\begin{equation}
\mathbb{E}\int_{t}^{T}\left\Vert \left( A_{\lambda }\right) ^{\varepsilon
}\left( Y_{s}^{\lambda ,\varepsilon }\right) \right\Vert _{-1}^{2}ds<C,
\label{margA}
\end{equation}%
for some constant $C$ independent of $\varepsilon $ and $\lambda $.\bigskip
\end{lemma}

\begin{lemma}
\label{lema2}Under the assumptions above, we have the following strong
convergences in $\varepsilon $ for the couple $\left( Y,Z\right) $
\begin{equation*}
\begin{aligned}
Z_{s}^{\lambda ,\varepsilon }&\rightarrow Z_{s}^{\lambda }\text{ strongly
in }L^{2}\left( \left( 0,T\right) \mathcal{\times }\Omega ; L_{2}\left(
L^{2}\left( \mathcal{O}\right) ,H^{-1}\left( \mathcal{O}\right) \right)
\right);\\
Y_{\tau }^{\lambda ,\varepsilon }&\rightarrow Y_{\tau }^{\lambda }\text{
strongly in }L^{2}\left( \Omega ;C\left( \left[ 0,T\right] ;H^{-1}\left(
\mathcal{O}\right) \right) \right) .
\end{aligned}
\end{equation*}
\end{lemma}

\bigskip

At this point, in order to get (\ref{margU}) we need to pass to the limit
for $\varepsilon \rightarrow 0$ in (\ref{margA}) .

\begin{lemma}
\label{lema3}Under the assumptions above, we have
\begin{equation}
\mathbb{E}\int_{t}^{T}\left\Vert A_{\lambda }\left( Y_{s}^{\lambda }\right)
\right\Vert _{-1}^{2}ds<C,  \label{margAlambda}
\end{equation}
for some constant $C$ independent of $\varepsilon $ and $\lambda $.\bigskip
\end{lemma}

\bigskip

On the other hand we have that%
\begin{equation*}
\mathbb{E}\left[\int_{t}^{T}\left\Vert h\left( Y_{s}^{\lambda },Z_{s}^{\lambda
}\right) \right\Vert _{-1}^{2}ds\right] \leq K^2\mathbb{E}\left[\int_{t}^{T}\left( \rho \left( Y^{\lambda }\right)
+\left\Vert Z^{\lambda }\right\Vert _{L_{2}\left( L^{2}\left( \mathcal{O}%
\right) ;H^{-1}\left( \mathcal{O}\right) \right) }\right)^2 ds\right]\leq C.
\end{equation*}

We can conclude now that
\begin{equation*}
\begin{aligned}
\mathbb{E}\left[\int_{t}^{T}\left\Vert U_{\lambda }\left( Y_{s}^{\lambda
},Z_{s}^{\lambda }\right) \right\Vert _{-1}^{2}ds\right] \leq& 2\mathbb{E}
\left[\int_{t}^{T}\left\Vert A_{\lambda }\left( Y_{s}^{\lambda }\right)
\right\Vert _{-1}^{2}ds\right]+2\mathbb{E}\left[\int_{t}^{T}\left\Vert h\left( Y^{\lambda
},Z^{\lambda }\right) \right\Vert _{-1}^{2}ds\right]\\
\leq& C,
\end{aligned}
\end{equation*}
which implies that $U_{\lambda }\left( Y_{s}^{\lambda },Z_{s}^{\lambda
}\right) $ converges in a weak sense to some function $\delta \in L^{2}\left( \left(
0,T\right) \times \Omega ;H^{-1}\left( \mathcal{O}\right) \right) $, i.e.
\begin{equation*}
U_{\lambda }\left( Y_{s}^{\lambda },Z_{s}^{\lambda }\right) \rightharpoonup \delta \text{
weakly in }L^{2}\left( \left( 0,T\right) \times \Omega ;H^{-1}\left(
\mathcal{O}\right) \right) .
\end{equation*}

Since $U_{\lambda }\left( y,z\right) =\Delta \Psi _{\lambda }\left( y\right)
+h\left( y,z\right) $ we have by the uniqueness of the limit that%
\begin{equation*}
U_{\lambda }\left( Y_{s}^{\lambda }, Z_{s}^{\lambda }\right) \rightharpoonup \Delta \eta +\phi
\text{ weakly in }L^{2}\left( \left( 0,T\right) \times \Omega ;H^{-1}\left(
\mathcal{O}\right) \right) .
\end{equation*}

\bigskip

\textbf{Step II (strong convergences \textbf{of }$Y_{t}^{\lambda }$\textbf{\
and }$Z_{s}^{\lambda }$\textbf{\ }in }$H^{-1}$\textbf{)}

\bigskip

For each $\lambda ,\ \eta >0$ we take the It\^{o} formula with $e^{\alpha
t}\left\Vert \cdot \right\Vert _{H^{-1}\left( \mathcal{O}\right) }^{2}$ for
the difference
\begin{equation}\label{dif}
\begin{aligned}
Y_{t}^{\lambda }-Y_{t}^{\eta } =&\int_{t}^{T}\left( \Delta \Psi _{\lambda
}\left( Y_{s}^{\lambda }\right) +h\left( Y_{s}^{\lambda },Z_{s}^{\lambda
}\right) \right) ds-\int_{t}^{T}\left( \Delta \Psi _{\eta }\left( Y_{s}^{\eta }\right)
+h\left( Y_{s}^{\eta },Z_{s}^{\eta }\right) \right) ds  \\
&-\int_{t}^{T}Z_{s}^{\lambda }dW_{s}+\int_{t}^{T}Z_{s}^{\eta }dW_{s},
\end{aligned}
\end{equation}
and we get%

\begin{equation}\label{Ito2}
\begin{aligned}
&\mathbb{E}\left[e^{\alpha t}\left\Vert Y_{t}^{\lambda }-Y_{t}^{\eta }\right\Vert
_{H^{-1}\left( \mathcal{O}\right) }^{2}\right]+\alpha \mathbb{E}\left[\int_{t}^{T}e^{%
\alpha s}\left\Vert Y_{s}^{\lambda }-Y_{s}^{\eta }\right\Vert _{H^{-1}\left(
\mathcal{O}\right) }^{2}ds\right]  \\
=&-2\mathbb{E}\left[\int_{t}^{T}e^{\alpha s}\left\langle \Psi _{\lambda }\left(
Y_{s}^{\lambda }\right) -\Psi _{\eta }\left( Y_{s}^{\eta }\right)
,Y_{s}^{\lambda }-Y_{s}^{\eta }\right\rangle _{L^{2}\left( \mathcal{O}%
\right) }ds\right]  \\
&+2\mathbb{E}\left[\int_{t}^{T}e^{\alpha s}\left\langle h\left( Y_{s}^{\lambda
},Z_{s}^{\lambda }\right) -h\left( Y_{s}^{\eta },Z_{s}^{\eta }\right)
,Y_{s}^{\lambda }-Y_{s}^{\eta }\right\rangle _{H^{-1}\left( \mathcal{O}%
\right) }ds\right]  \\
&-\mathbb{E}\left[\int_{t}^{T}e^{\alpha s}\left\Vert Z_{s}^{\lambda }-Z_{s}^{\eta
}\right\Vert _{L_{2}\left( L^{2}\left( \mathcal{O}\right) ,H^{-1}\left(
\mathcal{O}\right) \right) }^{2}ds\right].
\end{aligned}
\end{equation}
We compute%
\begin{equation}\label{calc1}
\begin{aligned}
&-2\mathbb{E}\left[\int_{t}^{T}e^{\alpha s}\left\langle \Psi _{\lambda }\left(
Y_{s}^{\lambda }\right) -\Psi _{\eta }\left( Y_{s}^{\eta }\right)
,Y_{s}^{\lambda }-Y_{s}^{\eta }\right\rangle _{L^{2}\left( \mathcal{O}%
\right) }ds \right] \\
=&-2\mathbb{E}\left[\int_{t}^{T}e^{\alpha s}\left\langle \Psi _{\lambda }\left(
Y_{s}^{\lambda }\right) -\Psi _{\eta }\left( Y_{s}^{\eta }\right) ,\lambda
\Psi _{\lambda }\left( Y_{s}^{\lambda }\right) -\eta \Psi _{\eta }\left(
Y_{s}^{\eta }\right) \right\rangle _{L^{2}\left( \mathcal{O}\right) }ds\right]
\\
&-2\mathbb{E}\left[\int_{t}^{T}e^{\alpha s}\left\langle \Psi \left( J_{\lambda
}\left( Y_{s}^{\lambda }\right) \right) -\Psi \left( J_{\eta }\left(
Y_{s}^{\eta }\right) \right) ,J_{\lambda }\left( Y_{s}^{\lambda }\right)
-J_{\eta }\left( Y_{s}^{\eta }\right) \right\rangle _{L^{2}\left( \mathcal{O}
\right) }ds \right] \\
\leq& C\left( \lambda +\eta \right)\mathbb{E}\left[ \int_{t}^{T}e^{\alpha s}\left(
\left\Vert \Psi _{\lambda }\left( Y_{s}^{\lambda }\right) \right\Vert
_{L^{2}\left( \mathcal{O}\right) }^{2}+\left\Vert \Psi _{\eta }\left(
Y_{s}^{\eta }\right) \right\Vert _{L^{2}\left( \mathcal{O}\right)
}^{2}\right) ds\right],
\end{aligned}
\end{equation}
and%
\begin{equation}\label{calc2}
\begin{aligned}
&2\int_{t}^{T}e^{\alpha s}\left\langle h\left( Y_{s}^{\lambda
},Z_{s}^{\lambda }\right) -h\left( Y_{s}^{\eta },Z_{s}^{\eta }\right)
,Y_{s}^{\lambda }-Y_{t}^{\eta }\right\rangle _{H^{-1}\left( \mathcal{O}%
\right) }ds  \\
\leq&2\int_{t}^{T}\rho \left( Y_{s}^{\lambda }\right) e^{\alpha
s}\left\Vert Y_{s}^{\lambda }-Y_{s}^{\eta }\right\Vert _{H^{-1}\left(
\mathcal{O}\right) }^{2}ds  \\
&+2\int_{t}^{T}\rho \left( Y_{s}^{\lambda }\right) e^{\alpha s}\left\Vert
Y_{s}^{\lambda }-Y_{t}^{\eta }\right\Vert _{H^{-1}\left( \mathcal{O}\right)
}\left\Vert Z_{s}^{\lambda }-Z_{s}^{\eta }\right\Vert _{L_{2}\left(
L^{2}\left( \mathcal{O}\right) ;H^{-1}\left( \mathcal{O}\right) \right) }ds
\\
\leq& 3\int_{t}^{T}\rho \left( Y_{s}^{\lambda }\right) e^{\alpha
s}\left\Vert Y_{s}^{\lambda }-Y_{s}^{\eta }\right\Vert _{H^{-1}\left(
\mathcal{O}\right) }^{2}ds+\int_{t}^{T}\rho \left( Y_{s}^{\lambda }\right) e^{\alpha s}\left\Vert
Z_{s}^{\lambda }-Z_{s}^{\eta }\right\Vert _{L_{2}\left( L^{2}\left( \mathcal{%
O}\right) ;H^{-1}\left( \mathcal{O}\right) \right) }^{2}ds.
\end{aligned}
\end{equation}

Going back to (\ref{Ito2}) we get that%
\begin{equation*}
\begin{aligned}
&\mathbb{E}\left[ e^{\alpha t}\left\Vert Y_{t}^{\lambda }-Y_{t}^{\eta
}\right\Vert _{H^{-1}\left( \mathcal{O}\right) }^{2}\right]+\mathbb{E}\int_{t}^{T}\left( \alpha -3\rho \left( Y_{s}^{\lambda }\right)
\right) e^{\alpha s}\left\Vert Y_{s}^{\lambda }-Y_{s}^{\eta }\right\Vert
_{H^{-1}\left( \mathcal{O}\right) }^{2}ds \\
&+\mathbb{E}\int_{t}^{T}\left( 1-\rho \left( Y_{s}^{\lambda }\right)
\right) e^{\alpha s}\left\Vert Z_{s}^{\lambda }-Z_{s}^{\eta }\right\Vert
_{L_{2}\left( L^{2}\left( \mathcal{O}\right) ,H^{-1}\left( \mathcal{O}%
\right) \right) }^{2}ds \\
\leq&C\left( \lambda +\eta \right) \int_{t}^{T}e^{\alpha s}\left(
\left\Vert \Psi _{\lambda }\left( Y_{s}^{\lambda }\right) \right\Vert
_{L^{2}\left( \mathcal{O}\right) }^{2}+\left\Vert \Psi _{\eta }\left(
Y_{s}^{\eta }\right) \right\Vert _{L^{2}\left( \mathcal{O}\right)
}^{2}\right) ds.
\end{aligned}
\end{equation*}

Since $\rho $ is bounded and smaller than $1$ and $\alpha $ is assumed to
be sufficiently large we get from (\ref{margAlambda}) that
\begin{equation*}
\begin{aligned}
&\mathbb{E}\left[ e^{\alpha t}\left\Vert Y_{t}^{\lambda }-Y_{t}^{\eta
}\right\Vert _{H^{-1}\left( \mathcal{O}\right) }^{2}\right]+\mathbb{E}\int_{t}^{T}e^{\alpha s}\left\Vert Z_{s}^{\lambda ,\varepsilon
}-Z_{s}^{\lambda ,\eta }\right\Vert _{L_{2}\left( L^{2}\left( \mathcal{O}%
\right) ,H^{-1}\left( \mathcal{O}\right) \right) }^{2}ds \\
\leq C\left( \varepsilon +\eta \right) .
\end{aligned}
\end{equation*}
We obtained that%
\begin{equation}
Z_{s}^{\lambda }\longrightarrow Z_{s}\text{ strongly in }L^{2}\left( \left(
0,T\right) \mathcal{\times }\Omega ; L_{2}\left( L^{2}\left( \mathcal{O}%
\right) ,H^{-1}\left( \mathcal{O}\right) \right) \right) .
\label{conv-tare-Z}
\end{equation}

By applying again the It\^{o} formula to (\ref{dif}) with the supremum
before the expectation we get that%

\begin{equation}\label{Ito3}
\begin{aligned}
&\mathbb{E}\left[ \underset{\tau \in \left[ t,T\right] }{\sup }e^{\alpha
\tau }\left\Vert Y_{\tau }^{\lambda }-Y_{\tau }^{\eta }\right\Vert
_{H^{-1}\left( \mathcal{O}\right) }^{2}\right]  +\alpha \mathbb{E}\left[\int_{t}^{T}e^{\alpha s}\left\Vert Y_{s}^{\lambda
}-Y_{s}^{\eta }\right\Vert _{H^{-1}\left( \mathcal{O}\right) }^{2}ds\right]
\\
=&-2\mathbb{E}\left[\int_{t}^{T}e^{\alpha s}\left\langle \Psi _{\lambda }\left(
Y_{s}^{\lambda }\right) -\Psi _{\eta }\left( Y_{s}^{\eta }\right)
,Y_{s}^{\lambda }-Y_{s}^{\eta }\right\rangle _{L^{2}\left( \mathcal{O}
\right) }ds \right] \\
&+2\mathbb{E}\left[\int_{t}^{T}e^{\alpha s}\left\langle h\left( Y_{s}^{\lambda
},Z_{s}^{\lambda }\right) -h\left( Y_{s}^{\eta },Z_{s}^{\eta }\right)
,Y_{s}^{\lambda }-Y_{s}^{\eta }\right\rangle _{H^{-1}\left( \mathcal{O}%
\right) }ds \right]\\
&-2\mathbb{E}\left[ \underset{\tau \in \left[ t,T\right] }{\sup }\int_{\tau
}^{T}e^{\alpha s}\left\langle Y_{s}^{\lambda }-Y_{s}^{\eta },\left(
Z_{s}^{\lambda }-Z_{s}^{\eta }\right) dW_{s}\right\rangle _{H^{-1}\left(
\mathcal{O}\right) }\right] \\
&-\mathbb{E}\left[\int_{t}^{T}e^{\alpha s}\left\Vert Z_{s}^{\lambda }-Z_{s}^{\eta
}\right\Vert _{L_{2}\left( L^{2}\left( \mathcal{O}\right) ,H^{-1}\left(
\mathcal{O}\right) \right) }^{2}ds\right].
\end{aligned}
\end{equation}

By the Burkholder-Davis-Gundy inequality we have that%

\begin{equation*}
\begin{aligned}
&\mathbb{E}\left[ \underset{\tau \in \left[ t,T\right] }{\sup }e^{\alpha
\tau }\left\vert \int_{\tau }^{T}\left\langle Y_{s}^{\lambda
}-Y_{s}^{\eta },\left( Z_{s}^{\lambda  }-Z_{s}^{\eta }\right) dW_{s}\right\rangle _{H^{-1}\left( \mathcal{O}\right)
}\right\vert \right] \\
\leq& C\mathbb{E}\left[ \int_{t}^{T}e^{\alpha s}\left\Vert Y_{s}^{\lambda
}-Y_{s}^{\eta }\right\Vert _{H^{-1}\left( \mathcal{O}%
\right) }^{2}\left\Vert Z_{s}^{\lambda }-Z_{s}^{\eta
}\right\Vert _{L_{2}\left( L^{2}\left( \mathcal{O}\right) ,H^{-1}\left(
\mathcal{O}\right) \right) }^{2}ds\right] ^{1/2} \\
\leq& C\left\{ \frac{1}{2C}\mathbb{E}[\underset{\tau \in \left[ t,T\right] }{
\sup }e^{\alpha \tau }\left\Vert Y_{s}^{\lambda
}-Y_{s}^{\eta }\right\Vert _{H^{-1}\left( \mathcal{O}\right)
}^{2}]\right.\\
&+\left.\frac{C}{2}\mathbb{E}\left[\int_{t}^{T}e^{\alpha s}\left\Vert
Z_{s}^{\lambda}-Z_{s}^{\eta }\right\Vert
_{L_{2}\left( L^{2}\left( \mathcal{O}\right) ,H^{-1}\left( \mathcal{O}%
\right) \right) }^{2}ds\right]\right\} .
\end{aligned}
\end{equation*}

Going back to (\ref{Ito3}) and using the arguments from (\ref{calc1}) and (%
\ref{calc2}) we get that%
\begin{equation*}
\begin{aligned}
&\frac{1}{2}\mathbb{E}\left[ \underset{\tau \in \left[ t,T\right] }{\sup }%
e^{\alpha \tau }\left\Vert Y_{\tau }^{\lambda }-Y_{\tau }^{\eta }\right\Vert
_{H^{-1}\left( \mathcal{O}\right) }^{2}\right]+\left( \alpha -3\rho \left( Y_{s}^{\lambda }\right) \right) \mathbb{E}
\int_{t}^{T}e^{\alpha s}\left\Vert Y_{s}^{\lambda }-Y_{s}^{\eta }\right\Vert
_{H^{-1}\left( \mathcal{O}\right) }^{2}ds \\
&+\mathbb{E}\int_{t}^{T}\left( 1-\rho \left( Y_{s}^{\lambda }\right)
\right) e^{\alpha s}\left\Vert Z_{s}^{\lambda }-Z_{s}^{\eta }\right\Vert
_{L_{2}\left( L^{2}\left( \mathcal{O}\right) ,H^{-1}\left( \mathcal{O}%
\right) \right) }^{2}ds \\
\leq &C\left( \lambda +\eta \right) \int_{t}^{T}e^{\alpha s}\left(
\left\Vert \Psi _{\lambda }\left( Y_{s}^{\lambda }\right) \right\Vert
_{L^{2}\left( \mathcal{O}\right) }^{2}+\left\Vert \Psi _{\eta }\left(
Y_{s}^{\eta }\right) \right\Vert _{L^{2}\left( \mathcal{O}\right)
}^{2}\right) ds \\
&+\frac{C^{2}}{2}\mathbb{E}\int_{t}^{T}e^{\alpha s}\left\Vert
Z_{s}^{\lambda  }-Z_{s}^{\eta }\right\Vert
_{L_{2}\left( L^{2}\left( \mathcal{O}\right) ,H^{-1}\left( \mathcal{O}%
\right) \right) }^{2}ds.
\end{aligned}
\end{equation*}

Finally, by using again (\ref{margAlambda}) and combining with (\ref%
{conv-tare-Z}) we get that%
\begin{equation}
Y^{\lambda }\rightarrow Y\text{ strongly in }L^{2}\left( \Omega ;C\left( %
\left[ 0,T\right] ;H^{-1}\left( \mathcal{O}\right) \right) \right),
\label{conv-Y-tare}
\end{equation}

and we keep in mind that we have%
\begin{equation*}
Z^{\lambda }\rightarrow Z\text{ strongly in }L^{2}\left( \left( 0,T\right)
\mathcal{\times }\Omega ; L_{2}\left( L^{2}\left( \mathcal{O}\right)
,H^{-1}\left( \mathcal{O}\right) \right) \right) .
\end{equation*}

\bigskip

\textbf{Step III (identification of }$\Psi \left( Y\right) =\eta $ and $%
h\left( Y,Z\right) =\phi $\textbf{)}

\bigskip

In order to get the identification, we define%
\begin{equation*}
r_{t}=2\int_{t}^{T}\left( \rho \left( v_{s}\right) +\rho ^{2}\left(
v_{s}\right) \right) ds,
\end{equation*}%
for $v\in L_{\mathcal{F}}^{\infty }\left( \Omega \times \left[ 0,T\right]
;V\right) $.

We apply the It\^{o} formula to%
\begin{equation*}
Y_{t}^{\lambda }=\xi +\int_{t}^{T}U_{\lambda }\left( Y_{s}^{\lambda
},Z_{s}^{\lambda }\right) ds-\int_{t}^{T}Z_{s}^{\lambda }dW_{s},
\end{equation*}%
and we get%
\begin{equation}\label{Ito-iden}
\begin{aligned}
&\mathbb{E}\left[ e^{r_{t}}\left\Vert Y_{t}^{\lambda }\right\Vert
_{H^{-1}\left( \mathcal{O}\right) }^{2}-e^{r_{T}}\left\Vert \xi \right\Vert
_{H^{-1}\left( \mathcal{O}\right) }^{2}\right]  \\
=&2\mathbb{E}\int_{t}^{T}e^{r_{s}}\left\langle U_{\lambda }\left(
Y_{s}^{\lambda },Z_{s}^{\lambda }\right) ,Y_{s}^{\lambda }\right\rangle
_{H^{-1}\left( \mathcal{O}\right) }ds-\mathbb{E}\int_{t}^{T}e^{r_{s}}\left%
\Vert Z_{s}^{\lambda }\right\Vert _{L_{2}\left( L^{2}\left( \mathcal{O}%
\right) ,H^{-1}\left( \mathcal{O}\right) \right) }^{2}ds  \\
&-2\mathbb{E}\int_{t}^{T}e^{r_{t}}\left( \rho \left( v_{s}\right) +\rho
^{2}\left( v_{s}\right) \right) \left\Vert Y_{s}^{\lambda }\right\Vert
_{H^{-1}\left( \mathcal{O}\right) }^{2}ds.
\end{aligned}
\end{equation}

We rewrite%
\begin{equation}\label{UU}
\begin{aligned}
&2\mathbb{E}\int_{t}^{T}e^{r_{s}}\left\langle U_{\lambda }\left(
Y_{s}^{\lambda },Z_{s}^{\lambda }\right) ,Y_{s}^{\lambda }\right\rangle
_{H^{-1}\left( \mathcal{O}\right) }ds  \\
=\ &2\mathbb{E}\int_{t}^{T}e^{r_{s}}\left\langle \Delta \Psi _{\lambda
}\left( Y_{s}^{\lambda }\right) ,Y_{s}^{\lambda }\right\rangle
_{H^{-1}\left( \mathcal{O}\right) }ds+2\mathbb{E}\int_{t}^{T}e^{r_{s}}\left%
\langle h\left( Y_{s}^{\lambda },Z_{s}^{\lambda }\right) ,Y_{s}^{\lambda
}\right\rangle _{H^{-1}\left( \mathcal{O}\right) }ds \\
\overset{Note}{=}&I_{1}+I_{2}.
\end{aligned}
\end{equation}

We have%
\begin{equation*}
\begin{aligned}
I_{1} =&2\mathbb{E}\int_{t}^{T}e^{r_{s}}\left\langle \Delta \Psi _{\lambda
}\left( Y_{s}^{\lambda }\right) ,Y_{s}^{\lambda }\right\rangle
_{H^{-1}\left( \mathcal{O}\right) }ds \\
=&2\mathbb{E}\int_{t}^{T}e^{r_{s}}\left\langle \Delta \Psi _{\lambda
}\left( Y_{s}^{\lambda }\right) ,Y_{s}^{\lambda }-J_{\lambda }\left(
Y_{s}^{\lambda }\right) +J_{\lambda }\left( Y_{s}^{\lambda }\right)
\right\rangle _{H^{-1}\left( \mathcal{O}\right) }ds \\
=&2\lambda \mathbb{E}\int_{t}^{T}e^{r_{s}}\left\langle \Delta \Psi
_{\lambda }\left( Y_{s}^{\lambda }\right) ,\Psi _{\lambda }\left(
Y_{s}^{\lambda }\right) \right\rangle _{H^{-1}\left( \mathcal{O}\right) }ds
\\
&+2\mathbb{E}\int_{t}^{T}e^{r_{s}}\left\langle \Delta \Psi _{\lambda
}\left( Y_{s}^{\lambda }\right) ,J_{\lambda }\left( Y_{s}^{\lambda }\right)
\right\rangle _{H^{-1}\left( \mathcal{O}\right) }ds \\
\leq &2\mathbb{E}\int_{t}^{T}e^{r_{s}}\left\langle \Delta \Psi _{\lambda
}\left( Y_{s}^{\lambda }\right) ,J_{\lambda }\left( Y_{s}^{\lambda }\right)
\right\rangle _{H^{-1}\left( \mathcal{O}\right) }ds\\
\leq &2\mathbb{E}\int_{t}^{T}e^{r_{s}}\left\langle \Delta \Psi
_{\lambda }\left( Y_{s}^{\lambda }\right) -\Delta \Psi \left( v_{s}\right)
+\Delta \Psi \left( v_{s}\right) ,J_{\lambda }\left( Y_{s}^{\lambda }\right)
\right\rangle _{H^{-1}\left( \mathcal{O}\right) }ds \\
=&2\mathbb{E}\int_{t}^{T}e^{r_{s}}\left\langle \Delta \Psi \left(
J_{\lambda }\left( Y_{s}^{\lambda }\right) \right) -\Delta \Psi \left(
v_{s}\right) ,J_{\lambda }\left( Y_{s}^{\lambda }\right) \right\rangle
_{H^{-1}\left( \mathcal{O}\right) }ds \\
&+2\mathbb{E}\int_{t}^{T}e^{r_{s}}\left\langle \Delta \Psi \left(
v_{s}\right) ,J_{\lambda }\left( Y_{s}^{\lambda }\right) \right\rangle
_{H^{-1}\left( \mathcal{O}\right) }ds\\
=&2\mathbb{E}\int_{t}^{T}e^{r_{s}}\left\langle \Delta \Psi \left(
J_{\lambda }\left( Y_{s}^{\lambda }\right) \right) -\Delta \Psi \left(
v_{s}\right) ,J_{\lambda }\left( Y_{s}^{\lambda }\right) -v_{s}\right\rangle
_{H^{-1}\left( \mathcal{O}\right) }ds \\
&+2\mathbb{E}\int_{t}^{T}e^{r_{s}}\left\langle \Delta \Psi \left(
J_{\lambda }\left( Y_{s}^{\lambda }\right) \right) -\Delta \Psi \left(
v_{s}\right) ,v_{s}\right\rangle _{H^{-1}\left( \mathcal{O}\right) }ds \\
&+2\mathbb{E}\int_{t}^{T}e^{r_{s}}\left\langle \Delta \Psi \left( v\right)
,J_{\lambda }\left( Y_{s}^{\lambda }\right) \right\rangle _{H^{-1}\left(
\mathcal{O}\right) }ds \\
\leq &2\mathbb{E}\int_{t}^{T}e^{r_{s}}\left\langle \Delta \Psi \left(
J_{\lambda }\left( Y_{s}^{\lambda }\right) \right) -\Delta \Psi \left(
v_{s}\right) ,v_{s}\right\rangle _{H^{-1}\left( \mathcal{O}\right) }ds \\
&+2\mathbb{E}\int_{t}^{T}e^{r_{s}}\left\langle \Delta \Psi \left( v\right)
,J_{\lambda }\left( Y_{s}^{\lambda }\right) \right\rangle _{H^{-1}\left(
\mathcal{O}\right) }ds.
\end{aligned}
\end{equation*}

On the other hand we have%
\begin{equation*}
\begin{aligned}
I_{2} =&2\mathbb{E}\int_{t}^{T}e^{r_{s}}\left\langle h\left( Y_{s}^{\lambda
},Z_{s}^{\lambda }\right) ,Y_{s}^{\lambda }\right\rangle _{H^{-1}\left(
\mathcal{O}\right) }ds \\
=&2\mathbb{E}\int_{t}^{T}e^{r_{s}}\left\langle h\left( Y_{s}^{\lambda
},Z_{s}^{\lambda }\right) ,Y_{s}^{\lambda }-v_{s}\right\rangle
_{H^{-1}\left( \mathcal{O}\right) }ds
 \\
&+2\mathbb{E}\int_{t}^{T}e^{r_{s}}\left%
\langle h\left( Y_{s}^{\lambda },Z_{s}^{\lambda }\right) ,v_{s}\right\rangle
_{H^{-1}\left( \mathcal{O}\right) }ds\\
=&2\mathbb{E}\int_{t}^{T}e^{r_{s}}\left\langle h\left( Y_{s}^{\lambda
},Z_{s}^{\lambda }\right) -h\left( v_{s},Z_{s}\right) ,Y_{s}^{\lambda
}-v_{s}\right\rangle _{H^{-1}\left( \mathcal{O}\right) }ds \\
&+2\mathbb{E}\int_{t}^{T}e^{r_{s}}\left\langle h\left( v_{s},Z_{s}\right)
,Y_{s}^{\lambda }-v_{s}\right\rangle _{H^{-1}\left( \mathcal{O}\right) }ds\\
&+ 2\mathbb{E}\int_{t}^{T}e^{r_{s}}\left\langle h\left( Y_{s}^{\lambda
},Z_{s}^{\lambda }\right) ,v_{s}\right\rangle _{H^{-1}\left( \mathcal{O}%
\right) }ds.
\end{aligned}
\end{equation*}
We continue with
\begin{equation*}
\begin{aligned}
I_2=&2\mathbb{E}\int_{t}^{T}e^{r_{s}}\left\langle h\left( Y_{s}^{\lambda
},Z_{s}^{\lambda }\right) -h\left( v_{s},Z_{s}\right) ,Y_{s}^{\lambda
}-v_{s}\right\rangle _{H^{-1}\left( \mathcal{O}\right) }ds \\
&+2\mathbb{E}\int_{t}^{T}e^{r_{s}}\left\langle h\left( Y_{s}^{\lambda
},Z_{s}^{\lambda }\right) -h\left( v_{s},Z_{s}\right) ,v_{s}\right\rangle
_{H^{-1}\left( \mathcal{O}\right) }ds\\
&+2\mathbb{E}\int_{t}^{T}e^{r_{s}}\left
\langle h\left( v_{s},Z_{s}\right) ,Y_{s}^{\lambda }\right\rangle
_{H^{-1}\left( \mathcal{O}\right) }ds\\
\leq &2\mathbb{E}\int_{t}^{T}e^{r_{s}}\left( \rho \left( v_{s}\right)
\left\Vert Y_{s}^{\lambda }-v_{s}\right\Vert _{H^{-1}\left( \mathcal{O}%
\right) }^{2}\right)ds
\\
&+2\mathbb{E}\int_{t}^{T}e^{r_{s}}\left(\rho \left( v_{s}\right) \left\Vert Y_{s}^{\lambda
}-v_{s}\right\Vert _{H^{-1}\left( \mathcal{O}\right) }\left\Vert
Z_{s}^{\lambda }-Z_{s}\right\Vert _{L_{2}\left( L^{2}\left( \mathcal{O}%
\right) ,H^{-1}\left( \mathcal{O}\right) \right) }\right) ds \\
&+2\mathbb{E}\int_{t}^{T}e^{r_{s}}\left\langle h\left( Y_{s}^{\lambda
},Z_{s}^{\lambda }\right) -h\left( v_{s},Z_{s}\right) ,v_{s}\right\rangle
_{H^{-1}\left( \mathcal{O}\right) }ds\\
&+2\mathbb{E}\int_{t}^{T}e^{r_{s}}\left%
\langle h\left( v_{s},Z_{s}\right) ,Y_{s}^{\lambda }\right\rangle
_{H^{-1}\left( \mathcal{O}\right) }ds\\
\leq &2\mathbb{E}\int_{t}^{T}e^{r_{s}}\left( \left( \rho \left(
v_{s}\right) +\rho ^{2}\left( v_{s}\right) \right) \left\Vert Y_{s}^{\lambda
}-v_{s}\right\Vert _{H^{-1}\left( \mathcal{O}\right) }^{2}+\left\Vert
Z_{s}^{\lambda }-Z_{s}\right\Vert _{L_{2}\left( L^{2}\left( \mathcal{O}%
\right) ,H^{-1}\left( \mathcal{O}\right) \right) }^{2}\right) ds \\
&+2\mathbb{E}\int_{t}^{T}e^{r_{s}}\left\langle h\left( Y_{s}^{\lambda
},Z_{s}^{\lambda }\right) -h\left( v_{s},Z_{s}\right) ,v_{s}\right\rangle
_{H^{-1}\left( \mathcal{O}\right) }ds\\
&+2\mathbb{E}\int_{t}^{T}e^{r_{s}}\left%
\langle h\left( v_{s},Z_{s}\right) ,Y_{s}^{\lambda }\right\rangle
_{H^{-1}\left( \mathcal{O}\right) }ds.
\end{aligned}
\end{equation*}

Going back to (\ref{UU}) we obtain%
\begin{equation*}
\begin{aligned}
&2\mathbb{E}\int_{t}^{T}e^{r_{s}}\left\langle U_{\lambda }\left(
Y_{s}^{\lambda },Z_{s}^{\lambda }\right) ,Y_{s}^{\lambda }\right\rangle
_{H^{-1}\left( \mathcal{O}\right) }ds \\
\leq &2\mathbb{E}\int_{t}^{T}e^{r_{s}}\left\langle \Delta \Psi \left(
J_{\lambda }\left( Y_{s}^{\lambda }\right) \right) -\Delta \Psi \left(
v_s\right) ,v_S\right\rangle _{H^{-1}\left( \mathcal{O}\right) }ds \\
&+2\mathbb{E}\int_{t}^{T}e^{r_{s}}\left\langle \Delta \Psi \left(
v_{s}\right) ,J_{\lambda }\left( Y_{s}^{\lambda }\right) \right\rangle
_{H^{-1}\left( \mathcal{O}\right) }ds \\
&+2\mathbb{E}\int_{t}^{T}e^{r_{s}}\left( \left( \rho \left( v_{s}\right)
+\rho ^{2}\left( v_{s}\right) \right) \left\Vert Y_{s}^{\lambda
}-v_{s}\right\Vert _{H^{-1}\left( \mathcal{O}\right) }^{2}+\left\Vert
Z_{s}^{\lambda }-Z_{s}\right\Vert _{L_{2}\left( L^{2}\left( \mathcal{O}%
\right) ,H^{-1}\left( \mathcal{O}\right) \right) }^{2}\right) ds \\
&+2\mathbb{E}\int_{t}^{T}e^{r_{s}}\left\langle h\left( Y_{s}^{\lambda
},Z_{s}^{\lambda }\right) -h\left( v_{s},Z_{s}\right) ,v_{s}\right\rangle
_{H^{-1}\left( \mathcal{O}\right) }ds+2\mathbb{E}\int_{t}^{T}e^{r_{s}}\left%
\langle h\left( v_{s},Z_{s}\right) ,Y_{s}^{\lambda }\right\rangle
_{H^{-1}\left( \mathcal{O}\right) }ds\\
=&2\mathbb{E}\int_{t}^{T}e^{r_{s}}\left\langle U_{\lambda }\left(
Y_{s}^{\lambda },Z_{s}^{\lambda }\right) -U\left( v_{s},Z_{s}\right)
,v_{s}\right\rangle _{H^{-1}\left( \mathcal{O}\right) }ds \\
&+2\mathbb{E}\int_{t}^{T}e^{r_{s}}\left\langle \Delta \Psi \left(
v_{s}\right) ,J_{\lambda }\left( Y_{s}^{\lambda }\right) \right\rangle
_{H^{-1}\left( \mathcal{O}\right) }ds \\
&+2\mathbb{E}\int_{t}^{T}e^{r_{s}}\left( \left( \rho \left( v_{s}\right)
+\rho ^{2}\left( v_{s}\right) \right) \left\Vert Y_{s}^{\lambda
}-v_{s}\right\Vert _{H^{-1}\left( \mathcal{O}\right) }^{2}+\left\Vert
Z_{s}^{\lambda }-Z_{s}\right\Vert _{L_{2}\left( L^{2}\left( \mathcal{O}%
\right) ,H^{-1}\left( \mathcal{O}\right) \right) }^{2}\right) ds \\
&+2\mathbb{E}\int_{t}^{T}e^{r_{s}}\left\langle h\left( v_{s},Z_{s}\right)
,Y_{s}^{\lambda }\right\rangle _{H^{-1}\left( \mathcal{O}\right) }ds.
\end{aligned}
\end{equation*}

We replace in (\ref{Ito-iden}) and we obtain that%
\begin{equation*}
\begin{aligned}
&\mathbb{E}\left[ e^{r_{t}}\left\Vert Y_{t}^{\lambda }\right\Vert
_{H^{-1}\left( \mathcal{O}\right) }^{2}-e^{r_{T}}\left\Vert \xi \right\Vert
_{H^{-1}\left( \mathcal{O}\right) }^{2}\right]  \\
\leq &2\mathbb{E}\int_{t}^{T}e^{r_{s}}\left\langle U_{\lambda }\left(
Y_{s}^{\lambda },Z_{s}^{\lambda }\right) -U\left( v_{s},Z_{s}\right)
,v_{s}\right\rangle _{H^{-1}\left( \mathcal{O}\right) }ds \\
&+2\mathbb{E}\int_{t}^{T}e^{r_{s}}\left\langle \Delta \Psi \left(
v_{s}\right) ,J_{\lambda }\left( Y_{s}^{\lambda }\right) \right\rangle
_{H^{-1}\left( \mathcal{O}\right) }ds \\
&+2\mathbb{E}\int_{t}^{T}e^{r_{s}}\left( \left( \rho \left( v_{s}\right)
+\rho ^{2}\left( v_{s}\right) \right) \left\Vert Y_{s}^{\lambda
}-v_{s}\right\Vert _{H^{-1}\left( \mathcal{O}\right) }^{2}+\left\Vert
Z_{s}^{\lambda }-Z_{s}\right\Vert _{L_{2}\left( L^{2}\left( \mathcal{O}%
\right) ,H^{-1}\left( \mathcal{O}\right) \right) }^{2}\right) ds \\
&+2\mathbb{E}\int_{t}^{T}e^{r_{s}}\left\langle h\left( v_{s},Z_{s}\right)
,Y_{s}^{\lambda }\right\rangle _{H^{-1}\left( \mathcal{O}\right) }ds \\
&-\mathbb{E}\int_{t}^{T}e^{r_{s}}\left\Vert Z_{s}^{\lambda }\right\Vert
_{L_{2}\left( L^{2}\left( \mathcal{O}\right) ,H^{-1}\left( \mathcal{O}%
\right) \right) }^{2}ds-2\mathbb{E}\int_{t}^{T}e^{r_{s}}\left( \rho \left(
v_{s}\right) +\rho ^{2}\left( v_{s}\right) \right) \left\Vert Y_{s}^{\lambda
}\right\Vert _{H^{-1}\left( \mathcal{O}\right) }^{2}ds\\
\leq &2\mathbb{E}\int_{t}^{T}e^{r_{s}}\left\langle U_{\lambda }\left(
Y_{s}^{\lambda },Z_{s}^{\lambda }\right) -U\left( v_{s},Z_{s}\right)
,v_{s}\right\rangle _{H^{-1}\left( \mathcal{O}\right) }ds \\
&+2\mathbb{E}\int_{t}^{T}e^{r_{s}}\left\langle \Delta \Psi \left(
v_{s}\right) ,J_{\lambda }\left( Y_{s}^{\lambda }\right) \right\rangle
_{H^{-1}\left( \mathcal{O}\right) }ds \\
&+2\mathbb{E}\int_{t}^{T}e^{r_{s}}\left( \left( \rho \left( v_{s}\right)
+\rho ^{2}\left( v_{s}\right) \right) \left\Vert Y_{s}^{\lambda
}-v_{s}\right\Vert _{H^{-1}\left( \mathcal{O}\right) }^{2}+\left\Vert
Z_{s}^{\lambda }-Z_{s}\right\Vert _{L_{2}\left( L^{2}\left( \mathcal{O}%
\right) ,H^{-1}\left( \mathcal{O}\right) \right) }^{2}\right) ds \\
&+2\mathbb{E}\int_{t}^{T}e^{r_{s}}\left\langle h\left( v_{s},Z_{s}\right)
,Y_{s}^{\lambda }\right\rangle _{H^{-1}\left( \mathcal{O}\right) }ds-\mathbb{%
E}\int_{t}^{T}e^{r_{s}}\left\Vert Z_{s}^{\lambda }\right\Vert _{L_{2}\left(
L^{2}\left( \mathcal{O}\right) ,H^{-1}\left( \mathcal{O}\right) \right)
}^{2}ds \\
&-2\mathbb{E}\int_{t}^{T}e^{r_{s}}\left( \rho \left( v_{s}\right) +\rho
^{2}\left( v_{s}\right) \right) \left( \left\Vert Y_{s}^{\lambda
}-v_{s}\right\Vert _{H^{-1}\left( \mathcal{O}\right) }^{2}+2\left\langle
Y_{s}^{\lambda },v_{s}\right\rangle _{H^{-1}\left( \mathcal{O}\right)
}-\left\Vert v_{s}\right\Vert _{H^{-1}\left( \mathcal{O}\right) }^{2}\right)
ds.
\end{aligned}
\end{equation*}

Passing to the limit for $\lambda \rightarrow 0$ we get that%
\begin{equation}\label{limita}
\begin{aligned}
&\mathbb{E}\left[ e^{r_{t}}\left\Vert Y_{t}\right\Vert _{H^{-1}\left(
\mathcal{O}\right) }^{2}-e^{r_{T}}\left\Vert \xi \right\Vert _{H^{-1}\left(
\mathcal{O}\right) }^{2}\right]  \\
\leq &2\mathbb{E}\int_{t}^{T}e^{r_{s}}\left\langle \delta -U\left(
v_{s},Z_{s}\right) ,v_{s}\right\rangle _{H^{-1}\left( \mathcal{O}\right)
}ds+2\mathbb{E}\int_{t}^{T}e^{r_{s}}\left\langle \Delta \Psi \left(
v_{s}\right) ,Y_{s}\right\rangle _{H^{-1}\left( \mathcal{O}\right) }ds\\
&+2\mathbb{E}\int_{t}^{T}e^{r_{s}}\left\langle h\left( v_{s},Z_{s}\right)
,Y_{s}\right\rangle _{H^{-1}\left( \mathcal{O}\right) }ds-\mathbb{E}%
\int_{t}^{T}e^{r_{s}}\left\Vert Z_{s}\right\Vert _{L_{2}\left( L^{2}\left(
\mathcal{O}\right) ,H^{-1}\left( \mathcal{O}\right) \right) }^{2}ds
\\
&-2\mathbb{E}\int_{t}^{T}e^{r_{s}}\left( \rho \left( v_{s}\right) +\rho
^{2}\left( v_{s}\right) \right) \left( 2\left\langle
Y_{s},v_{s}\right\rangle _{H^{-1}\left( \mathcal{O}\right) }-\left\Vert
v_{s}\right\Vert _{H^{-1}\left( \mathcal{O}\right) }^{2}\right) ds.
\end{aligned}
\end{equation}

On the other hand, by applying the It\^{o} formula to the limit equation, we
get that%
\begin{equation}\label{ito-limita}
\begin{aligned}
&\mathbb{E}\left[ e^{r_{t}}\left\Vert Y_{t}\right\Vert _{H^{-1}\left(
\mathcal{O}\right) }^{2}-e^{r_{T}}\left\Vert \xi \right\Vert _{H^{-1}\left(
\mathcal{O}\right) }^{2}\right]  \\
=&2\mathbb{E}\int_{t}^{T}e^{r_{s}}\left\langle \Delta \eta
,Y_{s}\right\rangle _{H^{-1}\left( \mathcal{O}\right) }ds+2\mathbb{E}%
\int_{t}^{T}e^{r_{s}}\left\langle \phi ,Y_{s}\right\rangle _{H^{-1}\left(
\mathcal{O}\right) }ds \\
&-\mathbb{E}\int_{t}^{T}e^{r_{s}}\left\Vert Z_{s}\right\Vert _{L_{2}\left(
L^{2}\left( \mathcal{O}\right) ,H^{-1}\left( \mathcal{O}\right) \right)
}^{2}ds-2\mathbb{E}\int_{t}^{T}e^{r_{s}}\left( \rho \left( v_{s}\right) +\rho
^{2}\left( v_{s}\right) \right) \left\Vert Y_{s}\right\Vert _{H^{-1}\left(
\mathcal{O}\right) }^{2}ds.
\end{aligned}
\end{equation}

By combining (\ref{limita}) and (\ref{ito-limita}) we get that%
\begin{equation*}
\begin{aligned}
&2\mathbb{E}\int_{t}^{T}e^{r_{s}}\left\langle \delta -U\left(
v_{s},Z_{s}\right) ,Y_{s}-v_{s}\right\rangle _{H^{-1}\left( \mathcal{O}%
\right) }ds \\
\leq &2\mathbb{E}\int_{t}^{T}e^{r_{s}}\left( \rho \left( v_{s}\right) +\rho
^{2}\left( v_{s}\right) \right) \left\Vert Y_{s}-v_{s}\right\Vert
_{H^{-1}\left( \mathcal{O}\right) }^{2}ds.
\end{aligned}
\end{equation*}

Taking $v_{s}=Y_{s}-\varepsilon \varphi w$ for $\varepsilon >0$ with $%
\varphi \in L^{\infty }\left( \left( 0,T\right) \times \Omega ;\mathbb{R}%
\right) $ and $w\in L^{m+1}\left( \mathcal{O}\right) $ we get%
\begin{equation}\label{ultima-ineg}
\begin{aligned}
&2\mathbb{E}\int_{t}^{T}e^{r_{s}}\left\langle \delta -U\left(
Y_{s}-\varepsilon \varphi w,Z_{s}\right) ,\varepsilon \varphi w\right\rangle
_{H^{-1}\left( \mathcal{O}\right) }ds  \\
\leq &2\mathbb{E}\int_{t}^{T}e^{r_{s}}\left( \rho \left( Y_{s}-\varepsilon
\varphi w\right) +\rho ^{2}\left( Y_{s}-\varepsilon \varphi w\right) \right)
\left\Vert \varepsilon \varphi w\right\Vert _{H^{-1}\left( \mathcal{O}%
\right) }^{2}ds.
\end{aligned}
\end{equation}

Then we divide by $\varepsilon $ and we want to let $\varepsilon \rightarrow
0$. To this purpose we take
\begin{equation*}
U\left( Y_{s}-\varepsilon \varphi w,Z_{s}\right) =\Delta \Psi \left(
Y_{s}-\varepsilon \varphi w\right) +h\left( Y_{s}-\varepsilon \varphi
w,Z_{s}\right),
\end{equation*}%
and we passe to the limit in
\begin{equation}
2\mathbb{E}\int_{t}^{T}e^{r_{s}}\left\langle \delta -\Delta \Psi \left(
Y_{s}-\varepsilon \varphi w\right) -h\left( Y_{s}-\varepsilon \varphi
w,Z_{s}\right) ,\varphi w\right\rangle _{H^{-1}\left( \mathcal{O}\right) }ds.
\label{termen}
\end{equation}

For the first term
\begin{equation*}
\begin{aligned}
&2\mathbb{E}\int_{t}^{T}e^{r_{s}}\left\langle \delta -\Delta \Psi \left(
Y_{s}-\varepsilon \varphi w\right) ,\varphi w\right\rangle _{H^{-1}\left(
\mathcal{O}\right) }ds \\
=&2\mathbb{E}\int_{t}^{T}e^{r_{s}}\left\langle \left( -\Delta \right)
^{-1}\delta +\Psi \left( Y_{s}-\varepsilon \varphi w\right) ,\varphi
w\right\rangle _{L^{2}\left( \mathcal{O}\right) }ds.
\end{aligned}
\end{equation*}
We recall that $Y_{s}\in L^{m+1}\left( \Omega \times \left( 0,T\right)
\times \mathcal{O}\right) $ and $\left\{ \Psi \left( Y_{s}-\varepsilon
\varphi w\right) \right\} $ is bounded in $L^{\frac{m+1}{m}}\left( \Omega
\times \left( 0,T\right) \times \mathcal{O}\right) $.

Since
\begin{equation*}
\Psi :L^{m+1}\left( \Omega \times \left( 0,T\right) \times \mathcal{O}%
\right) \rightarrow L^{\frac{m+1}{m}}\left( \Omega \times \left( 0,T\right)
\times \mathcal{O}\right),
\end{equation*}%
is a maximal monotone operator and
\begin{eqnarray*}
Y_{s}-\varepsilon \varphi w &\rightarrow &Y_{s}\text{ strongly in }%
L^{m+1}\left( \Omega \times \left( 0,T\right) \times \mathcal{O}\right);  \\
\Psi \left( Y_{s}-\varepsilon \varphi w\right)  &\rightharpoonup &\varkappa
\text{ weakly in }L^{\frac{m+1}{m}}\left( \Omega \times \left( 0,T\right)
\times \mathcal{O}\right),
\end{eqnarray*}%
we can conclude that
\begin{equation*}
\Psi \left( Y_{s}-\varepsilon \varphi w\right) \rightarrow \Psi \left(
Y_{s}\right) \text{ weakly in }L^{\frac{m+1}{m}}\left( \Omega \times \left(
0,T\right) \times \mathcal{O}\right) .
\end{equation*}

Keeping in mind the assumptions on $h$ we can pass to the limit in (\ref%
{termen}) and we obtain
\begin{equation*}
2\mathbb{E}\int_{t}^{T}e^{r_{s}}\left\langle \delta -\Delta \Psi \left(
Y_{s}\right) -h\left( Y_{s},Z_{s}\right) ,\varphi w\right\rangle
_{H^{-1}\left( \mathcal{O}\right) }ds.
\end{equation*}

On the other hand, since the functions $\rho $\ are bounded, after letting $%
\varepsilon \rightarrow 0$ in (\ref{ultima-ineg})\ we get that%
\begin{equation*}
\mathbb{E}\int_{t}^{T}e^{r_{s}}\left\langle \delta -U\left(
Y_{s},Z_{s}\right) ,\varphi w\right\rangle _{H^{-1}\left( \mathcal{O}\right)
}ds\leq 0,
\end{equation*}%
which leads to $\delta =U\left( Y,Z\right) =\Delta \Psi \left( Y\right)
+h\left( Y,Z\right) $ since $\varphi $ and $w$ were chosen arbitrarily.
\end{proof}

\bigskip

\section{Appendix}

\bigskip

\begin{proof}[Proof of Lemma \protect\ref{lema1}]
Keeping in mind the assumption $c)$ and the estimate (\ref{estim-vechi}) we
can apply the It\^{o} formula in $L^{2}\left( \mathcal{O}\right) $ to the
equation%
\begin{equation}
Y_{t}^{\lambda ,\varepsilon }=\xi +\int_{t}^{T}\left( -\left( A_{\lambda
}\right) ^{\varepsilon }\left( Y_{s}^{\lambda ,\varepsilon }\right) +h\left(
Y_{s}^{\lambda ,\varepsilon },Z_{s}^{\lambda ,\varepsilon }\right) \right)
ds-\int_{t}^{T}Z_{s}^{\lambda ,\varepsilon }dW_{s} , \label{sol-epsilon}
\end{equation}%
and we get%
\begin{equation}\label{estim1}
\begin{aligned}
e^{\alpha t}\left\Vert Y_{t}^{\lambda ,\varepsilon }\right\Vert
_{L^{2}\left( \mathcal{O}\right) }^{2} =&e^{\alpha T}\left\Vert \xi
\right\Vert _{L^{2}\left( \mathcal{O}\right) }^{2}  \\
&-2\int_{t}^{T}e^{\alpha s}\left\langle \left( A_{\lambda }\right)
^{\varepsilon }\left( Y_{s}^{\lambda ,\varepsilon }\right) ,Y_{s}^{\lambda
,\varepsilon }\right\rangle _{L^{2}\left( \mathcal{O}\right) }ds  \\
&+2\int_{t}^{T}e^{\alpha s}\left\langle h\left( Y_{s}^{\lambda ,\varepsilon
},Z_{s}^{\lambda ,\varepsilon }\right) ,Y_{s}^{\lambda ,\varepsilon
}\right\rangle _{L^{2}\left( \mathcal{O}\right) }ds  \\
&-2\int_{t}^{T}e^{\alpha s}\left\langle Y_{s}^{\lambda ,\varepsilon
},Z_{s}^{\lambda ,\varepsilon }dW_{s}\right\rangle _{L^{2}\left( \mathcal{O}%
\right) }  \\
&-\int_{t}^{T}e^{\alpha s}\left\Vert Z_{s}^{\lambda ,\varepsilon
}\right\Vert _{L_{2}\left( L^{2}\left( \mathcal{O}\right) ,L^{2}\left(
\mathcal{O}\right) \right) }^{2}ds.
\end{aligned}
\end{equation}
From the equation of the resolvent%
\begin{equation*}
J_{\varepsilon }\left( Y_{s}^{\lambda ,\varepsilon }\right) -\varepsilon
\Delta \Psi _{\lambda }\left( J_{\varepsilon }\left( Y_{s}^{\lambda
,\varepsilon }\right) \right) =Y_{s}^{\lambda ,\varepsilon },
\end{equation*}%
we can easily see that
\begin{equation*}
\begin{aligned}
&\left\langle \left( A_{\lambda }\right) ^{\varepsilon }\left(
Y_{s}^{\lambda ,\varepsilon }\right) ,Y_{s}^{\lambda ,\varepsilon
}\right\rangle _{L^{2}\left( \mathcal{O}\right) } \\
=&\left\langle A_{\lambda }\left( J_{\varepsilon }\left( Y_{s}^{\lambda
,\varepsilon }\right) \right) ,J_{\varepsilon }\left( Y_{s}^{\lambda
,\varepsilon }\right) -J_{\varepsilon }\left( Y_{s}^{\lambda ,\varepsilon
}\right) +Y_{s}^{\lambda ,\varepsilon }\right\rangle _{L^{2}\left( \mathcal{O%
}\right) } \\
=&\left\langle A_{\lambda }\left( J_{\varepsilon }\left( Y_{s}^{\lambda
,\varepsilon }\right) \right) ,J_{\varepsilon }\left( Y_{s}^{\lambda
,\varepsilon }\right) \right\rangle _{L^{2}\left( \mathcal{O}\right)
}+\varepsilon \left\Vert Y_{s}^{\lambda ,\varepsilon }-J_{\varepsilon
}\left( Y_{s}^{\lambda ,\varepsilon }\right) \right\Vert _{L^{2}\left(
\mathcal{O}\right) }^{2}\geq 0,
\end{aligned}
\end{equation*}
and keeping in mind the assumption $c),$ we obtain from (\ref{estim1}) that%
\begin{equation}\label{estim2}
\begin{aligned}
&\left\Vert Y_{t}^{\lambda ,\varepsilon }\right\Vert _{L^{2}\left( \mathcal{%
O}\right) }^{2}+\int_{t}^{T}e^{\alpha \left( s-t\right) }\left\langle
A_{\lambda }\left( J_{\varepsilon }\left( Y_{s}^{\lambda ,\varepsilon
}\right) \right) ,J_{\varepsilon }\left( Y_{s}^{\lambda ,\varepsilon
}\right) \right\rangle _{L^{2}\left( \mathcal{O}\right) }ds
\\
\leq &e^{\alpha \left( T-t\right) }\left\Vert \xi \right\Vert _{L^{2}\left(
\mathcal{O}\right) }^{2}+\left( 2K-\alpha \right) \int_{t}^{T}e^{\alpha
\left( s-t\right) }\left\Vert Y_{s}^{\lambda ,\varepsilon }\right\Vert
_{L^{2}\left( \mathcal{O}\right) }^{2}ds  \\
&-2\int_{t}^{T}e^{\alpha \left( s-t\right) }\left\langle Y_{s}^{\lambda
,\varepsilon },Z_{s}^{\lambda ,\varepsilon }dW_{s}\right\rangle _{2}
+\left( 2\gamma -1\right) \int_{t}^{T}e^{\alpha \left( s-t\right)
}\left\Vert Z_{s}^{\lambda ,\varepsilon }\right\Vert _{L_{2}\left(
L^{2}\left( \mathcal{O}\right) ,L^{2}\left( \mathcal{O}\right) \right)
}^{2}ds.
\end{aligned}
\end{equation}
By taking the expectation in the previous relation we obtain for $\alpha >2K$
that
\begin{equation*}
\begin{aligned}
&\mathbb{E}\left\Vert Y_{t}^{\lambda ,\varepsilon }\right\Vert
_{L^{2}\left( \mathcal{O}\right) }^{2}+\mathbb{E}\int_{t}^{T}e^{\alpha
\left( s-t\right) }\left\langle A_{\lambda }\left( J_{\varepsilon }\left(
Y_{s}^{\lambda ,\varepsilon }\right) \right) ,J_{\varepsilon }\left(
Y_{s}^{\lambda ,\varepsilon }\right) \right\rangle _{L^{2}\left( \mathcal{O}%
\right) }ds \\
&+\left( 1-2\gamma \right) \mathbb{E}\int_{t}^{T}e^{\alpha \left(
s-t\right) }\left\Vert Z_{s}^{\lambda ,\varepsilon }\right\Vert
_{L_{2}\left( L^{2}\left( \mathcal{O}\right) ,L^{2}\left( \mathcal{O}\right)
\right) }^{2}ds \\
\leq &\mathbb{E}e^{\alpha \left( T-t\right) }\left\Vert \xi \right\Vert
_{L^{2}\left( \mathcal{O}\right) }^{2}+\left( 2K-\alpha \right) \mathbb{E}%
\int_{t}^{T}e^{\alpha \left( s-t\right) }\left\Vert Y_{s}^{\lambda
,\varepsilon }\right\Vert _{L^{2}\left( \mathcal{O}\right) }^{2}ds,
\end{aligned}
\end{equation*}
which leads to%
\begin{equation}\label{marg-Z}
\begin{aligned}
&\mathbb{E}\left\Vert Y_{t}^{\lambda ,\varepsilon }\right\Vert
_{L^{2}\left( \mathcal{O}\right) }^{2}+\mathbb{E}\int_{t}^{T}e^{\alpha
\left( s-t\right) }\left\langle A_{\lambda }\left( J_{\varepsilon }\left(
Y_{s}^{\lambda ,\varepsilon }\right) \right) ,J_{\varepsilon }\left(
Y_{s}^{\lambda ,\varepsilon }\right) \right\rangle _{L^{2}\left( \mathcal{O}%
\right) }ds  \\
&+\left( 1-2\gamma \right) \mathbb{E}\int_{t}^{T}e^{\alpha \left(
s-t\right) }\left\Vert Z_{s}^{\lambda ,\varepsilon }\right\Vert
_{L_{2}\left( L^{2}\left( \mathcal{O}\right) ,L^{2}\left( \mathcal{O}\right)
\right) }^{2}ds   \\
\leq &C\mathbb{E}e^{\alpha \left( T-t\right) }\left\Vert \xi \right\Vert
_{L^{2}\left( \mathcal{O}\right) }^{2},
\end{aligned}
\end{equation}
for some constant $C$ independent of $\lambda $ and $\varepsilon $. Keep in
mind that $\gamma $ is assumed to be as small as necessary and therefore $%
\left( 1-2\gamma \right) >0.$

Since $\Psi _{\lambda }$ is Lipschitz there exist a constant $\gamma >0$
such that $\dfrac{d}{dr}\Psi _{\lambda }^{-1}\left( r\right) \geq \gamma$, $%
\forall r\in \mathbb{R}$, we have
\begin{equation*}
\begin{aligned}
&\left\langle A_{\lambda }\left( J_{\varepsilon }\left( Y_{s}^{\lambda
,\varepsilon }\right) \right) ,J_{\varepsilon }\left( Y_{s}^{\lambda
,\varepsilon }\right) \right\rangle _{L^{2}\left( \mathcal{O}\right) } \\
=&\left\langle -\Delta \Psi _{\lambda }\left( J_{\varepsilon }\left(
Y_{s}^{\lambda ,\varepsilon }\right) \right) ,\Psi _{\lambda }^{-1}\left(
\Psi _{\lambda }\left( J_{\varepsilon }\left( Y_{s}^{\lambda ,\varepsilon
}\right) \right) \right) \right\rangle _{L^{2}\left( \mathcal{O}\right) } \\
=&\left\langle \nabla \Psi _{\lambda }\left( J_{\varepsilon }\left(
Y_{s}^{\lambda ,\varepsilon }\right) \right) ,\nabla \Psi _{\lambda
}^{-1}\left( \Psi _{\lambda }\left( J_{\varepsilon }\left( Y_{s}^{\lambda
,\varepsilon }\right) \right) \right) \right\rangle _{L^{2}\left( \mathcal{O}%
\right) } \\
\geq &\gamma \left\Vert \nabla \Psi _{\lambda }\left( J_{\varepsilon
}\left( Y_{s}^{\lambda ,\varepsilon }\right) \right) \right\Vert
_{L^{2}\left( \mathcal{O}\right) }^{2} =\left\Vert \left( A_{\lambda }\right) ^{\varepsilon }\left(
Y_{s}^{\lambda ,\varepsilon }\right) \right\Vert _{H^{-1}\left( \mathcal{O}%
\right) }^{2}.
\end{aligned}
\end{equation*}

Going back to (\ref{marg-Z}) we get that
\begin{equation}\label{marg-A}
\begin{aligned}
&\mathbb{E}\left\Vert Y_{t}^{\lambda ,\varepsilon }\right\Vert
_{L^{2}\left( \mathcal{O}\right) }^{2}+\mathbb{E}\int_{t}^{T}e^{\alpha
\left( s-t\right) }\left\Vert \left( A_{\lambda }\right) ^{\varepsilon
}\left( Y_{s}^{\lambda ,\varepsilon }\right) \right\Vert _{H^{-1}\left(
\mathcal{O}\right) }^{2}ds  \\
&+\left( 1-2\gamma \right) \mathbb{E}\int_{t}^{T}e^{\alpha \left(
s-t\right) }\left\Vert Z_{s}^{\lambda ,\varepsilon }\right\Vert
_{L_{2}\left( L^{2}\left( \mathcal{O}\right) ,L^{2}\left( \mathcal{O}\right)
\right) }^{2}ds  \\
\leq &C\mathbb{E}e^{\alpha \left( T-t\right) }\left\Vert \xi \right\Vert
_{L^{2}\left( \mathcal{O}\right) }^{2}.
\end{aligned}
\end{equation}

We go back to (\ref{estim2}) and we take the supremum over time followed by
the expectation and we get%
\begin{equation}\label{estim3}
\begin{aligned}
&\mathbb{E}\left[ \underset{\tau \in \left[ t,T\right] }{\sup }e^{\alpha
\tau }\left\Vert Y_{\tau }^{\lambda ,\varepsilon }\right\Vert _{L^{2}\left(
\mathcal{O}\right) }^{2}\right] +\mathbb{E}\int_{t}^{T}e^{\alpha
s}\left\Vert \left( A_{\lambda }\right) ^{\varepsilon }\left( Y_{s}^{\lambda
,\varepsilon }\right) \right\Vert _{H^{-1}\left( \mathcal{O}\right) }^{2}ds
\\
&+\left( 1-2\gamma \right) \mathbb{E}\int_{t}^{T}e^{\alpha s}\left\Vert
Z_{s}^{\lambda ,\varepsilon }\right\Vert _{L_{2}\left( L^{2}\left( \mathcal{O%
}\right) ,L^{2}\left( \mathcal{O}\right) \right) }^{2}ds \\
\leq &\mathbb{E}e^{\alpha T}\left\Vert \xi \right\Vert _{L^{2}\left(
\mathcal{O}\right) }^{2}+\left( 2K-\alpha \right) \mathbb{E}%
\int_{t}^{T}e^{\alpha s}\left\Vert Y_{s}^{\lambda ,\varepsilon }\right\Vert
_{L^{2}\left( \mathcal{O}\right) }^{2}ds  -2\mathbb{E}\left[ \underset{\tau \in \left[ t,T\right] }{\sup }\int_{\tau
}^{T}e^{\alpha s}\left\langle Y_{s}^{\lambda ,\varepsilon },Z_{s}^{\lambda
,\varepsilon }dW_{s}\right\rangle _{2}\right] .
\end{aligned}
\end{equation}

By the Burkholder-Davis-Gundy inequality we have that%
\begin{equation*}
\begin{aligned}
&\mathbb{E}\left[ \underset{\tau \in \left[ t,T\right] }{\sup }e^{\alpha
\tau }\left\vert \int_{\tau }^{T}\left\langle Y_{s}^{\lambda ,\varepsilon
},Z_{s}^{\lambda ,\varepsilon }dW_{s}\right\rangle _{2}\right\vert \right] \\
\leq &C\mathbb{E}\left[ \int_{t}^{T}e^{\alpha s}\left\Vert Y_{s}^{\lambda
,\varepsilon }\right\Vert _{L^{2}\left( \mathcal{O}\right) }^{2}\left\Vert
Z_{s}^{\lambda ,\varepsilon }\right\Vert _{L_{2}\left( L^{2}\left( \mathcal{O%
}\right) ,L^{2}\left( \mathcal{O}\right) \right) }^{2}ds\right] ^{1/2} \\
\leq &C\left\{ \frac{1}{2C}\mathbb{E}\underset{\tau \in \left[ t,T\right] }{%
\sup }e^{\alpha \tau }\left\Vert Y_{\tau }^{\lambda ,\varepsilon
}\right\Vert _{L^{2}\left( \mathcal{O}\right) }^{2}+\frac{C}{2}\mathbb{E}%
\int_{t}^{T}e^{\alpha s}\left\Vert Z_{s}^{\lambda ,\varepsilon }\right\Vert
_{L_{2}\left( L^{2}\left( \mathcal{O}\right) ,L^{2}\left( \mathcal{O}\right)
\right) }^{2}ds\right\},
\end{aligned}
\end{equation*}
which leads by taking the expectation in both sides of (\ref{estim3}) to%
\begin{equation*}
\begin{aligned}
&\frac{1}{2}\mathbb{E}\left[ \underset{\tau \in \left[ t,T\right] }{\sup }%
e^{\alpha \tau }\left\Vert Y_{\tau }^{\lambda ,\varepsilon }\right\Vert
_{2}^{2}\right] +\left( 1-2\gamma \right) \mathbb{E}\int_{t}^{T}e^{\alpha
s}\left\Vert Z_{s}^{\lambda ,\varepsilon }\right\Vert _{L_{2}\left(
L^{2}\left( \mathcal{O}\right) ,L^{2}\left( \mathcal{O}\right) \right)
}^{2}ds \\
\leq &e^{\alpha T}\mathbb{E}\left\Vert \xi \right\Vert _{2}^{2}+\left(
2K-\alpha \right) \mathbb{E}\int_{t}^{T}e^{\alpha s}\left\Vert
Y_{s}^{\lambda ,\varepsilon }\right\Vert _{L^{2}\left( \mathcal{O}\right)
}^{2}ds+\frac{C^{2}}{2}\mathbb{E}\int_{t}^{T}e^{\alpha s}\left\Vert
Z_{s}^{\lambda ,\varepsilon }\right\Vert _{L_{2}\left( L^{2}\left( \mathcal{O%
}\right) ,L^{2}\left( \mathcal{O}\right) \right) }^{2}ds,
\end{aligned}
\end{equation*}
and by using (\ref{marg-Z}) we have that
\begin{equation*}
\begin{aligned}
\frac{1}{2}\mathbb{E}\left[ \underset{\tau \in \left[ t,T\right] }{\sup }%
e^{\alpha \tau }\left\Vert Y_{\tau }^{\lambda ,\varepsilon }\right\Vert
_{2}^{2}\right]\leq e^{\alpha T}\mathbb{E}\left\Vert \xi \right\Vert _{2}^{2}+C+\left(
2K-\alpha \right) \mathbb{E}\int_{t}^{T}e^{\alpha s}\left\Vert
Y_{s}^{\lambda ,\varepsilon }\right\Vert _{L^{2}\left( \mathcal{O}\right)
}^{2}ds,
\end{aligned}
\end{equation*}
where $C$ is a generic constant which changes from one line to the other
remaining always independent of $\lambda $ and $\varepsilon .$

Finally, since for $\alpha >2K$, we obtain that
\begin{equation}
\frac{1}{2}\mathbb{E}\left[ \underset{\tau \in \left[ t,T\right] }{\sup }%
e^{\alpha \tau }\left\Vert Y_{\tau }^{\lambda ,\varepsilon }\right\Vert
_{2}^{2}\right] \leq Ce^{\alpha T}\mathbb{E}\left\Vert \xi \right\Vert
_{2}^{2}.  \label{marg-Y}
\end{equation}

From (\ref{marg-Z}) and (\ref{marg-Y}) we get that
\begin{equation*}
\begin{aligned}
&\frac{1}{2}\mathbb{E}\left[ \underset{\tau \in \left[ t,T\right] }{\sup }%
e^{\alpha \tau }\left\Vert Y_{\tau }^{\lambda ,\varepsilon }\right\Vert
_{2}^{2}\right] +\mathbb{E}\int_{t}^{T}e^{\alpha s}\left\Vert \left(
A_{\lambda }\right) ^{\varepsilon }\left( Y_{s}^{\lambda ,\varepsilon
}\right) \right\Vert _{H^{-1}\left( \mathcal{O}\right) }^{2}ds \\
&+\left( 1-2\gamma \right) \mathbb{E}\int_{t}^{T}e^{\alpha s}\left\Vert
Z_{s}^{\lambda ,\varepsilon }\right\Vert _{L_{2}\left( L^{2}\left( \mathcal{O%
}\right) ,L^{2}\left( \mathcal{O}\right) \right) }^{2}ds \\
\leq &C\mathbb{E}\left\Vert \xi \right\Vert _{L^{2}\left( \mathcal{O}%
\right) }^{2}.
\end{aligned}
\end{equation*}
\end{proof}

\bigskip

\begin{proof}[Proof of Lemma \protect\ref{lema2}]

Let $\varepsilon ,\ \eta >0$. For each $\lambda $ fixed, we take the It\^{o}
formula with $e^{\alpha t}\left\Vert \cdot \right\Vert _{H^{-1}\left(
\mathcal{O}\right) }^{2}$ for the difference
\begin{equation}\label{dif1}
\begin{aligned}
Y_{t}^{\lambda ,\varepsilon }-Y_{t}^{\lambda ,\eta } =&\int_{t}^{T}\left(
-\left( A_{\lambda }\right) ^{\varepsilon }\left( Y_{s}^{\lambda
,\varepsilon }\right) +h\left( Y_{s}^{\lambda ,\varepsilon },Z_{s}^{\lambda
,\varepsilon }\right) \right) ds  \\
&-\int_{t}^{T}\left( -\left( A_{\lambda }\right) ^{\eta }\left(
Y_{s}^{\lambda ,\eta }\right) +h\left( Y_{s}^{\lambda ,\eta },Z_{s}^{\lambda
,\eta }\right) \right) ds  \\
&-\int_{t}^{T}Z_{s}^{\lambda ,\varepsilon
}dW_{s}+\int_{t}^{T}Z_{s}^{\lambda ,\eta }dW_{s} ,
\end{aligned}
\end{equation}
and we obtain after taking the expectation that
\begin{equation}\label{ito5}
\begin{aligned}
&\mathbb{E}e^{\alpha t}\left\Vert Y_{t}^{\lambda ,\varepsilon
}-Y_{t}^{\lambda ,\eta }\right\Vert _{H^{-1}\left( \mathcal{O}\right)
}^{2}+\alpha \mathbb{E}\int_{t}^{T}e^{\alpha s}\left\Vert Y_{s}^{\lambda
,\varepsilon }-Y_{s}^{\lambda ,\eta }\right\Vert _{H^{-1}\left( \mathcal{O}%
\right) }^{2}ds  \\
=&-2\mathbb{E}\int_{t}^{T}e^{\alpha s}\left\langle \left( A_{\lambda
}\right) ^{\varepsilon }\left( Y_{s}^{\lambda ,\varepsilon }\right) -\left(
A_{\lambda }\right) ^{\eta }\left( Y_{s}^{\lambda ,\eta }\right)
,Y_{s}^{\lambda ,\varepsilon }-Y_{s}^{\lambda ,\eta }\right\rangle
_{H^{-1}\left( \mathcal{O}\right) }ds  \\
&+2\mathbb{E}\int_{t}^{T}e^{\alpha s}\left\langle h\left( Y_{s}^{\lambda
,\varepsilon },Z_{s}^{\lambda ,\varepsilon }\right) -h\left( Y_{s}^{\lambda
,\eta },Z_{s}^{\lambda ,\eta }\right) ,Y_{s}^{\lambda ,\varepsilon
}-Y_{s}^{\lambda ,\eta }\right\rangle _{H^{-1}\left( \mathcal{O}\right) }ds \\
&-\mathbb{E}\int_{t}^{T}e^{\alpha s}\left\Vert Z_{s}^{\lambda ,\varepsilon
}-Z_{s}^{\lambda ,\eta }\right\Vert _{L_{2}\left( L^{2}\left( \mathcal{O}%
\right) ,H^{-1}\left( \mathcal{O}\right) \right) }^{2}ds.
\end{aligned}
\end{equation}
We compute%
\begin{equation*}
\begin{aligned}
&-\int_{t}^{T}e^{\alpha s}\left\langle \left( A_{\lambda }\right)
^{\varepsilon }\left( Y_{s}^{\lambda ,\varepsilon }\right) -\left(
A_{\lambda }\right) ^{\eta }\left( Y_{s}^{\lambda ,\eta }\right)
,Y_{s}^{\lambda ,\varepsilon }-Y_{s}^{\lambda ,\eta }\right\rangle
_{H^{-1}\left( \mathcal{O}\right) }ds \\
=&-\int_{t}^{T}e^{\alpha s}\left\langle \left( A_{\lambda }\right)
^{\varepsilon }\left( Y_{s}^{\lambda ,\varepsilon }\right) -\left(
A_{\lambda }\right) ^{\eta }\left( Y_{s}^{\lambda ,\eta }\right)
,\varepsilon \left( A_{\lambda }\right) ^{\varepsilon }\left( Y_{s}^{\lambda
,\varepsilon }\right) -\eta \left( A_{\lambda }\right) ^{\eta }\left(
Y_{s}^{\lambda ,\eta }\right) \right\rangle _{H^{-1}\left( \mathcal{O}%
\right) }ds \\
&-\int_{t}^{T}e^{\alpha s}\left\langle \left( A_{\lambda }\right)
^{\varepsilon }\left( Y_{s}^{\lambda ,\varepsilon }\right) -\left(
A_{\lambda }\right) ^{\eta }\left( Y_{s}^{\lambda ,\eta }\right)
,J_{\varepsilon }\left( Y_{s}^{\lambda ,\varepsilon }\right) -J_{\eta
}\left( Y_{s}^{\lambda ,\eta }\right) \right\rangle _{H^{-1}\left( \mathcal{O%
}\right) }ds \\
\leq &C\left( \varepsilon +\eta \right) \int_{t}^{T}e^{\alpha s}\left(
\left\Vert \left( A_{\lambda }\right) ^{\varepsilon }\left( Y_{s}^{\lambda
,\varepsilon }\right) \right\Vert _{H^{-1}\left( \mathcal{O}\right)
}^{2}+\left\Vert \left( A_{\lambda }\right) ^{\eta }\left( Y_{s}^{\lambda
,\eta }\right) \right\Vert _{H^{-1}\left( \mathcal{O}\right) }^{2}\right) ds,
\end{aligned}
\end{equation*}
and%

\begin{equation*}
\begin{aligned}
&2\int_{t}^{T}e^{\alpha s}\left\langle h\left( Y_{s}^{\lambda ,\varepsilon
},Z_{s}^{\lambda ,\varepsilon }\right) -h\left( Y_{s}^{\lambda ,\eta
},Z_{s}^{\lambda ,\eta }\right) ,Y_{s}^{\lambda ,\varepsilon
}-Y_{t}^{\lambda ,\eta }\right\rangle _{H^{-1}\left( \mathcal{O}\right) }ds
\\
\leq &2\int_{t}^{T}\rho \left( Y_{s}^{\lambda ,\varepsilon }\right)
e^{\alpha s}\left\Vert Y_{s}^{\lambda ,\varepsilon }-Y_{s}^{\lambda ,\eta
}\right\Vert _{H^{-1}\left( \mathcal{O}\right) }^{2}ds \\
&+2\int_{t}^{T}\rho \left( Y_{s}^{\lambda ,\varepsilon }\right) e^{\alpha
s}\left\Vert Y_{s}^{\lambda ,\varepsilon }-Y_{t}^{\lambda ,\eta }\right\Vert
_{H^{-1}\left( \mathcal{O}\right) }\left\Vert Z_{s}^{\lambda ,\varepsilon
}-Z_{s}^{\lambda ,\eta }\right\Vert _{L_{2}\left( L^{2}\left( \mathcal{O}%
\right) ;H^{-1}\left( \mathcal{O}\right) \right) }ds \\
\leq &2\int_{t}^{T}\rho \left( Y_{s}^{\lambda ,\varepsilon }\right)
e^{\alpha s}\left\Vert Y_{s}^{\lambda ,\varepsilon }-Y_{s}^{\lambda ,\eta
}\right\Vert _{H^{-1}\left( \mathcal{O}\right) }^{2}ds \\
&+2\int_{t}^{T}\rho \left( Y_{s}^{\lambda ,\varepsilon }\right) e^{\alpha
s}\left\Vert Z_{s}^{\lambda ,\varepsilon }-Z_{s}^{\lambda ,\eta }\right\Vert
_{L_{2}\left( L^{2}\left( \mathcal{O}\right) ;H^{-1}\left( \mathcal{O}%
\right) \right) }^{2}ds.
\end{aligned}
\end{equation*}

Going back to (\ref{ito5}) we get that%
\begin{equation*}
\begin{aligned}
&\mathbb{E}\left[ e^{\alpha t}\left\Vert Y_{t}^{\lambda ,\varepsilon
}-Y_{t}^{\lambda ,\eta }\right\Vert _{H^{-1}\left( \mathcal{O}\right) }^{2}%
\right] \\
&+\mathbb{E}\int_{t}^{T}\left( \alpha -2\rho \left( Y_{s}^{\lambda
,\varepsilon }\right) \right) e^{\alpha s}\left\Vert Y_{s}^{\lambda
,\varepsilon }-Y_{s}^{\lambda ,\eta }\right\Vert _{H^{-1}\left( \mathcal{O}%
\right) }^{2}ds \\
&+\mathbb{E}\int_{t}^{T}\left( 1-2\rho \left( Y_{s}^{\lambda ,\varepsilon
}\right) \right) e^{\alpha s}\left\Vert Z_{s}^{\lambda ,\varepsilon
}-Z_{s}^{\lambda ,\eta }\right\Vert _{L_{2}\left( L^{2}\left( \mathcal{O}%
\right) ,H^{-1}\left( \mathcal{O}\right) \right) }^{2}ds \\
\leq &C\left( \varepsilon +\eta \right) \int_{t}^{T}e^{\alpha s}\left(
\left\Vert \left( A_{\lambda }\right) ^{\varepsilon }\left( Y_{s}^{\lambda
,\varepsilon }\right) \right\Vert _{H^{-1}\left( \mathcal{O}\right)
}^{2}+\left\Vert \left( A_{\lambda }\right) ^{\eta }\left( Y_{s}^{\lambda
,\eta }\right) \right\Vert _{H^{-1}\left( \mathcal{O}\right) }^{2}\right) ds.
\end{aligned}
\end{equation*}

Since $\rho $ is bounded and smaller than $1/2$ and $\alpha $ is assumed to
be sufficiently large we get that%
\begin{equation*}
\begin{aligned}
&\mathbb{E}\left[ e^{\alpha t}\left\Vert Y_{t}^{\lambda ,\varepsilon
}-Y_{t}^{\lambda ,\eta }\right\Vert _{H^{-1}\left( \mathcal{O}\right) }^{2}%
\right]+\mathbb{E}\int_{t}^{T}e^{\alpha s}\left\Vert Z_{s}^{\lambda ,\varepsilon
}-Z_{s}^{\lambda ,\eta }\right\Vert _{L_{2}\left( L^{2}\left( \mathcal{O}%
\right) ,H^{-1}\left( \mathcal{O}\right) \right) }^{2}ds \\
\leq &C\left( \varepsilon +\eta \right) .
\end{aligned}
\end{equation*}

We obtained that%
\begin{equation}
Z_{s}^{\lambda ,\varepsilon }\longrightarrow Z_{s}^{\lambda }\text{ strongly
in }L^{2}\left( \left( 0,T\right) \mathcal{\times }\Omega \times L_{2}\left(
L^{2}\left( \mathcal{O}\right) ,H^{-1}\left( \mathcal{O}\right) \right)
\right) .  \label{conv-Z-tare}
\end{equation}\\
By applying again the It\^{o} formula to (\ref{dif1}) with the supremum
before the expectation we get that%
\begin{equation}\label{estim-dif2}
\begin{aligned}
&\mathbb{E}\left[ \underset{\tau \in \left[ t,T\right] }{\sup }e^{\alpha
\tau }\left\Vert Y_{\tau }^{\lambda ,\varepsilon }-Y_{\tau }^{\lambda ,\eta
}\right\Vert _{H^{-1}\left( \mathcal{O}\right) }^{2}\right]+\alpha \mathbb{E}\int_{t}^{T}e^{\alpha s}\left\Vert Y_{s}^{\lambda
,\varepsilon }-Y_{s}^{\lambda ,\eta }\right\Vert _{H^{-1}\left( \mathcal{O}%
\right) }^{2}ds  \\
=&-2\mathbb{E}\int_{t}^{T}e^{\alpha s}\left\langle \left( A_{\lambda
}\right) ^{\varepsilon }\left( Y_{s}^{\lambda ,\varepsilon }\right) -\left(
A_{\lambda }\right) ^{\eta }\left( Y_{s}^{\lambda ,\eta }\right)
,Y_{s}^{\lambda ,\varepsilon }-Y_{s}^{\lambda ,\eta }\right\rangle
_{H^{-1}\left( \mathcal{O}\right) }ds  \\
&+2\mathbb{E}\int_{t}^{T}e^{\alpha s}\left\langle h\left( Y_{s}^{\lambda
,\varepsilon },Z_{s}^{\lambda ,\varepsilon }\right) -h\left( Y_{s}^{\lambda
,\eta },Z_{s}^{\lambda ,\eta }\right) ,Y_{s}^{\lambda ,\varepsilon
}-Y_{s}^{\lambda ,\eta }\right\rangle _{H^{-1}\left( \mathcal{O}\right) }ds\\
&-2\mathbb{E}\left[ \underset{\tau \in \left[ t,T\right] }{\sup }\int_{\tau
}^{T}e^{\alpha s}\left\langle Y_{s}^{\lambda ,\varepsilon }-Y_{s}^{\lambda
,\eta },\left( Z_{s}^{\lambda ,\varepsilon }-Z_{s}^{\lambda ,\eta }\right)
dW_{s}\right\rangle _{H^{-1}\left( \mathcal{O}\right) }\right]  \\
&-\mathbb{E}\int_{t}^{T}e^{\alpha s}\left\Vert Z_{s}^{\lambda ,\varepsilon
}-Z_{s}^{\lambda ,\eta }\right\Vert _{L_{2}\left( L^{2}\left( \mathcal{O}%
\right) ,H^{-1}\left( \mathcal{O}\right) \right) }^{2}ds.
\end{aligned}
\end{equation}

By the Burkholder-Davis-Gundy inequality we have that%
\begin{equation*}
\begin{aligned}
&\mathbb{E}\left[ \underset{\tau \in \left[ t,T\right] }{\sup }e^{\alpha
\tau }\left\vert \int_{\tau }^{T}\left\langle Y_{s}^{\lambda ,\varepsilon
}-Y_{s}^{\lambda ,\eta },\left( Z_{s}^{\lambda ,\varepsilon }-Z_{s}^{\lambda
,\eta }\right) dW_{s}\right\rangle _{H^{-1}\left( \mathcal{O}\right)
}\right\vert \right] \\
\leq &C\mathbb{E}\left[ \int_{t}^{T}e^{\alpha s}\left\Vert Y_{s}^{\lambda
,\varepsilon }-Y_{s}^{\lambda ,\eta }\right\Vert _{H^{-1}\left( \mathcal{O}%
\right) }^{2}\left\Vert Z_{s}^{\lambda ,\varepsilon }-Z_{s}^{\lambda ,\eta
}\right\Vert _{L_{2}\left( L^{2}\left( \mathcal{O}\right) ,H^{-1}\left(
\mathcal{O}\right) \right) }^{2}ds\right] ^{1/2} \\
\leq &C\left\{ \frac{1}{2C}\mathbb{E}\underset{\tau \in \left[ t,T\right] }{%
\sup }e^{\alpha \tau }\left\Vert Y_{s}^{\lambda ,\varepsilon
}-Y_{s}^{\lambda ,\eta }\right\Vert _{H^{-1}\left( \mathcal{O}\right)
}^{2}\right. \\
&\left. +\frac{C}{2}\mathbb{E}\int_{t}^{T}e^{\alpha s}\left\Vert
Z_{s}^{\lambda ,\varepsilon }-Z_{s}^{\lambda ,\eta }\right\Vert
_{L_{2}\left( L^{2}\left( \mathcal{O}\right) ,H^{-1}\left( \mathcal{O}%
\right) \right) }^{2}ds\right\} .
\end{aligned}
\end{equation*}

Going back to (\ref{estim-dif2}) and using again the arguments from above we
get that%
\begin{equation*}
\begin{aligned}
&\frac{1}{2}\mathbb{E}\left[ \underset{\tau \in \left[ t,T\right] }{\sup }%
e^{\alpha \tau }\left\Vert Y_{\tau }^{\lambda ,\varepsilon }-Y_{\tau
}^{\lambda ,\eta }\right\Vert _{H^{-1}\left( \mathcal{O}\right) }^{2}\right]
+\left( \alpha -2\rho \left( Y_{s}^{\lambda ,\varepsilon }\right) \right)
\mathbb{E}\int_{t}^{T}e^{\alpha s}\left\Vert Y_{s}^{\lambda ,\varepsilon
}-Y_{s}^{\lambda ,\eta }\right\Vert _{H^{-1}\left( \mathcal{O}\right) }^{2}ds
\\
\leq &C\left( \varepsilon +\eta \right) \mathbb{E}\int_{t}^{T}e^{\alpha
s}\left( \left\Vert \left( A_{\lambda }\right) ^{\varepsilon }\left(
Y_{s}^{\lambda ,\varepsilon }\right) \right\Vert _{H^{-1}\left( \mathcal{O}%
\right) }^{2}+\left\Vert \left( A_{\lambda }\right) ^{\eta }\left(
Y_{s}^{\lambda ,\eta }\right) \right\Vert _{H^{-1}\left( \mathcal{O}\right)
}^{2}\right) ds \\
&+\frac{C^{2}}{2}\mathbb{E}\int_{t}^{T}e^{\alpha s}\left\Vert
Z_{s}^{\lambda ,\varepsilon }-Z_{s}^{\lambda ,\eta }\right\Vert
_{L_{2}\left( L^{2}\left( \mathcal{O}\right) ,H^{-1}\left( \mathcal{O}%
\right) \right) }^{2}ds.
\end{aligned}
\end{equation*}

Finally, by using (\ref{conv-Z-tare}) we get that%
\begin{equation}
Y_{\tau }^{\lambda ,\varepsilon }\rightarrow Y_{\tau }^{\lambda }\text{
strongly in }L^{2}\left( \Omega ;C\left( \left[ 0,T\right] ;H^{-1}\left(
\mathcal{O}\right) \right) \right) .  \label{conv-YY-tare}
\end{equation}

Since we obtained (\ref{conv-YY-tare}) and (\ref{conv-Z-tare}) the proof of
the Lemma is now complete.
\end{proof}

\bigskip
\bigskip

\begin{proof}[Proof of Lemma \protect\ref{lema3}]

We first note that
\begin{equation*}
\underset{\varepsilon \rightarrow 0}{\lim }J_{\varepsilon }\left( Y^{\lambda
,\varepsilon }\right) =Y^{\lambda }\text{ in }L^{2}\left( 0,T;L^{2}\left(
\Omega ;H^{-1}\left( \mathcal{O}\right) \right) \right),
\end{equation*}%
because%
\begin{equation*}
\begin{aligned}
&\mathbb{E}\int_{t}^{T}\left\Vert J_{\varepsilon }\left( Y^{\lambda
,\varepsilon }\right) -Y^{\lambda }\right\Vert _{-1}ds \\
=&\mathbb{E}\int_{t}^{T}\left\Vert J_{\varepsilon }\left( Y^{\lambda
,\varepsilon }\right) -Y^{\lambda ,\varepsilon }+Y^{\lambda ,\varepsilon
}-Y^{\lambda }\right\Vert _{-1}ds \\
\leq &\mathbb{E}\int_{t}^{T}\left\Vert J_{\varepsilon }\left( Y^{\lambda
,\varepsilon }\right) -Y^{\lambda ,\varepsilon }\right\Vert _{-1}ds+\mathbb{E%
}\int_{t}^{T}\left\Vert Y^{\lambda ,\varepsilon }-Y^{\lambda }\right\Vert
_{-1}ds \\
\leq &\varepsilon \mathbb{E}\int_{t}^{T}\left\Vert \left( A_{\lambda
}\right) ^{\varepsilon }\left( Y_{s}^{\lambda ,\varepsilon }\right)
\right\Vert _{-1}ds+\mathbb{E}\int_{t}^{T}\left\Vert Y^{\lambda ,\varepsilon
}-Y^{\lambda }\right\Vert _{-1}ds\rightarrow 0,\ \varepsilon \rightarrow 0.
\end{aligned}
\end{equation*}

Since $\Psi _{\lambda }$ is Lipschitz in $\mathbb{R}$ we have that $\left\{
\Psi _{\lambda }\left( J_{\varepsilon }\left( Y_{\lambda }^{\varepsilon
}\right) \right) \right\} _{\varepsilon }$ is a Cauchy sequence in $%
L^{2}\left( \left( 0,T\right) \times \Omega \times \mathcal{O}\right) $,
i.e.
\begin{equation*}
\begin{aligned}
&\mathbb{E}\int_{t}^{T}\int_{\mathcal{O}}\left\vert \Psi _{\lambda }\left(
J_{\varepsilon }\left( Y^{\lambda ,\varepsilon }\right) \right) -\Psi
_{\lambda }\left( J_{\eta }\left( Y^{\lambda ,\eta }\right) \right)
\right\vert ^{2}d\xi ds \\
\leq &\frac{1}{\lambda }\mathbb{E}\int_{t}^{T}\int_{\mathcal{O}}\left\vert
\Psi _{\lambda }\left( J_{\varepsilon }\left( Y^{\lambda ,\varepsilon
}\right) \right) -\Psi _{\lambda }\left( J_{\eta }\left( Y^{\lambda ,\eta
}\right) \right) \right\vert \left\vert J_{\varepsilon }\left( Y^{\lambda
,\varepsilon }\right) -J_{\eta }\left( Y^{\lambda ,\eta }\right) \right\vert
d\xi ds \\
\leq &\frac{2}{\lambda }\left[ \mathbb{E}\int_{t}^{T}\left( \left\Vert \Psi
_{\lambda }\left( J_{\varepsilon }\left( Y^{\lambda ,\varepsilon }\right)
\right) \right\Vert _{H_{0}^{1}\left( \mathcal{O}\right) }^{2}+\left\Vert
\Psi _{\lambda }\left( J_{\eta }\left( Y^{\lambda ,\eta }\right) \right)
\right\Vert _{H_{0}^{1}\left( \mathcal{O}\right) }^{2}\right) d\xi ds\right]
^{1/2} \\
&\times \frac{2}{\lambda }\left[
\mathbb{E}\int_{t}^{T}\left\Vert J_{\varepsilon }\left( Y^{\lambda
,\varepsilon }\right) -J_{\eta }\left( Y^{\lambda ,\eta }\right) \right\Vert
_{-1}^{2}d\xi ds\right] ^{1/2} \\
\leq &C\left( \lambda \right) \left[ \mathbb{E}\int_{t}^{T}\left\Vert
J_{\varepsilon }\left( Y^{\lambda ,\varepsilon }\right) -J_{\eta }\left(
Y^{\lambda ,\eta }\right) \right\Vert _{-1}^{2}d\xi ds\right] ^{1/2}\rightarrow 0,\ \varepsilon \rightarrow 0.
\end{aligned}
\end{equation*}

Since%
\begin{equation*}
\underset{\varepsilon \rightarrow 0}{\lim }\mathbb{E}\int_{t}^{T}\int_{%
\mathcal{O}}\left\vert \Psi _{\lambda }\left( J_{\varepsilon }\left(
Y^{\lambda ,\varepsilon }\right) \right) -\Psi _{\lambda }\left( J_{\eta
}\left( Y^{\lambda ,\eta }\right) \right) \right\vert ^{2}d\xi ds=0,
\end{equation*}%
\begin{equation*}
\underset{\varepsilon \rightarrow 0}{\lim }J_{\varepsilon }\left( Y^{\lambda
,\varepsilon }\right) =Y^{\lambda }\text{ weakly in }L^{2}\left( \left(
0,T\right) \times \Omega \times \mathcal{O}\right).
\end{equation*}%
and $\Psi _{\lambda }$ is a maximal monotone operator in $L^{2}\left( \left(
0,T\right) \times \Omega \times \mathcal{O}\right) $ we can conclude that
\begin{equation*}
\Psi _{\lambda }\left( J_{\varepsilon }\left( Y^{\lambda ,\varepsilon
}\right) \right) \underset{\varepsilon \rightarrow 0}{\rightarrow }\Psi
_{\lambda }\left( Y^{\lambda }\right) \text{ strongly in }L^{2}\left( \left(
0,T\right) \times \Omega \times \mathcal{O}\right),
\end{equation*}%
and therefore we can obtain also that%
\begin{equation*}
\left( A_{\lambda }\right) ^{\varepsilon }\left( Y_{s}^{\lambda ,\varepsilon
}\right) \underset{\varepsilon \rightarrow 0}{\rightarrow }A_{\lambda
}\left( Y_{s}^{\lambda }\right) \text{ weakly in }L^{2}\left( \left(
0,T\right) \times \Omega ;H^{-1}\left( \mathcal{O}\right) \right) .
\end{equation*}

Now we obtain that
\begin{equation*}
\mathbb{E}\int_{t}^{T}\left\Vert A_{\lambda }\left( Y_{s}^{\lambda }\right)
\right\Vert _{-1}^{2}ds<C.
\end{equation*}
\end{proof}

\noindent \textbf{Acknowledgements}
D.G. aknowledges financial support from National Sciences and Engineering Research Council (NSERC), Canada,
Grant/Award Number: RGPIN-2025-03963.

\end{document}